\documentclass[11pt, reqno]{amsart}
\usepackage{latexsym,amsmath,amssymb, amsthm, mathtools}
\usepackage{cases, verbatim}
\usepackage[active]{srcltx}
\usepackage{esint}
\usepackage[colorlinks]{hyperref}
\usepackage{hypernat}
\usepackage{xcolor}
\usepackage{float}
\usepackage[normalem]{ulem}
\usepackage{cancel}
\usepackage{subcaption}

\usepackage[T1]{fontenc}
\usepackage{mathscinet}

\hypersetup{
    bookmarks=true,         
    unicode=false,          
    pdftoolbar=true,        
    pdfmenubar=true,        
    pdffitwindow=false,     
    colorlinks=false,       
    linkcolor=red,          
    linkbordercolor=red,
    citecolor=green,        
    citebordercolor=green,
    filecolor=magenta,      
    urlcolor=cyan,           
   urlbordercolor={1 1 1}  
}

\usepackage[margin=1.5in]{geometry}

\newtheorem{thm}{Theorem}
\newtheorem{lem}[thm]{Lemma}
\newtheorem{prop}[thm]{Proposition}

\theoremstyle{remark}
\newtheorem{rmk}[thm]{Remark}
\newtheorem{example}[thm]{Example}

\theoremstyle{definition}
\newtheorem{defi}[thm]{Definition}

\numberwithin{thm}{section} 
\numberwithin{equation}{section}

\makeatletter

\newcommand{\Rmnum}[1]{\expandafter\@slowromancap\romannumeral #1@}
\makeatother

\def\Om{\Omega}

\def\R{{\mathbb R}}
\def\X{{\mathbf X}}

\newcommand{\cH}{\mathcal H}

\newcommand{\pO}{\partial\Omega}
\newcommand{\Oba}{\overline{\Omega}}

\newcommand{\vep}{\varepsilon}

\newcommand{\ul}{\underline}

\newcommand{\bpm}{\begin{pmatrix}}
\newcommand{\epm}{\end{pmatrix}}
\newcommand{\mo}{\operatorname{Mod}}

\newcommand{\beq}{\begin{equation}}
\newcommand{\eeq}{\end{equation}}

\def\diam{{\rm diam\,}}

\def\loc{{\rm loc}}

plus 6pt minus 12pt
\title[Discontinuous eikonal equations]{Discontinuous eikonal equations\\ in metric measure spaces}

\author[Q. Liu]{Qing Liu}
\address{Qing Liu, Geometric Partial Differential Equations Unit, Okinawa Institute of Science and Technology Graduate University, 1919-1, Tancha Onna-son, Okinawa 904-0495, Japan\\ Email: {\tt qing.liu@oist.jp}}

\author[N. Shanmugalingam]{Nageswari Shanmugalingam}
\address{Nageswari Shanmugalingam, Department of Mathematical Sciences, University of Cincinnati, 
P.O.Box 210025, Cincinnati, OH 45221-0025, USA\\ Email: {\tt shanmun@uc.edu}}

\author[X. Zhou]{Xiaodan Zhou}
\address{Xiaodan Zhou, Analysis on Metric Spaces Unit, Okinawa Institute of Science and Technology G
raduate University, 1919-1 Tancha, Onna-son, 
Okinawa 904-0495, Japan\\
Email: {\tt xiaodan.zhou@oist.jp}}

\date{\today}

\begin{document}

\begin{abstract}
In this paper, we study the eikonal equation in metric measure spaces, where the inhomogeneous term is allowed to be discontinuous, unbounded and merely $p$-integrable in the domain with a finite $p$. For continuous eikonal equations, it is known that the notion of Monge solutions is equivalent to the standard definition of viscosity solutions.  Generalizing the notion of Monge solutions in our setting, we establish uniqueness and existence results for the associated Dirichlet boundary problem. The key in our approach is to adopt a new metric, based on the optimal control interpretation, which integrates the discontinuous term and converts the eikonal equation to a standard continuous form. We also discuss the H\"older continuity of the unique solution with respect to the original metric under regularity assumptions on the space and the inhomogeneous term. 
\end{abstract}

\subjclass[2010]{35R15, 49L25, 35F30, 35D40}
\keywords{eikonal equation, metric spaces, discontinuous inhomogeneous term, viscosity solutions}

\maketitle

\section{Introduction}
\subsection{Background and motivation}
In this paper we study the eikonal equation in a metric measure space $(\X, d, \mu)$ with a possibly discontinuous inhomogenous term, where $\mu$ is assumed to be a nonnegative, locally finite Borel measure. We assume that $\X$ is a complete length space and $\Omega\subsetneq\X$ is a bounded domain. We consider 
\begin{equation}\label{eikonal}
      |\nabla u|(x)=f(x)\quad \text{$x\in \Omega$}
\end{equation}
with the Dirichlet boundary condition
\begin{equation}\label{dirichlet}
u=g \quad \text{on $\partial\Omega$,}
\end{equation}
where $f$ is a given positive Borel measurable function in $\overline{\Omega}$ and $g$ is a given bounded function on $\partial \Omega$. Further assumptions on $f$ and $g$ will be introduced later. 

Well-posedness of the eikonal equation and more general Hamilton-Jacobi equations is a classical topic. The theory of viscosity solutions provides a successful tool to solve such fully nonlinear PDEs in the Euclidean space. It is well known that when $\Omega\subset \R^n$ and $f\in C(\Oba)$ satisfies the lower bound condition
\begin{equation}\label{f-lower}
\alpha:=\inf_{\Omega}f >0,
\end{equation}
there exists a unique viscosity solution $u\in C(\Oba)$ of \eqref{eikonal} with Dirichlet data
\eqref{dirichlet} under appropriate regularity condition on $g$. The uniqueness result is obtained by establishing a comparison principle while the existence of solutions can be proved via several different approaches including Perron's method and a control-based formula.

Let us give more details on the connection with the optimal control theory, as it largely motivates our exploration and plays a central role in this work. It is widely known \cite{Libook, BCBook} that, under appropriate conditions on $f$ and $g$, the unique viscosity solution $u$ of \eqref{eikonal}
 with the Dirichlet boundary condition~\eqref{dirichlet} in the Euclidean space can be expressed by 
\begin{equation}\label{control formula}
u(x)=\inf\left\{g(y)+\int_\gamma f\, ds:\ \text{$\gamma$ is a curve in $\Oba$ joining $x\in \Oba$ and $y\in \partial\Omega$}\right\}.
\end{equation}
This formula represents the so-called value function of an optimal control problem, where one tries to move a point from $x\in \Oba$ toward the boundary and minimize the total payoff comprising a running cost $f$ integrated along the trajectory and a terminal cost $g$ at the first hit on the boundary. See \cite{IM, HCHZ} and references therein for more recent developments on representation formulas for general Hamilton-Jacobi equations. 

The representation formula not only provides a clear understanding of the solution and promotes various control-related practical applications, but also suggests possible generalization of the theory for a possibly discontinuous function $f$. 
In fact, much progress has been made in this direction in the Euclidean space. The study on discontinuous Hamilton-Jacobi equations in the Euclidean space was initiated in the work \cite{I0} and later found applications in geometric optics with different layers, image processing, shape from shading, etc. Further development in different contexts to treat discontinuous Hamiltonians including the time-dependent case can be found in \cite{Tou, NeSu, Ostr, Sor, CaSi1, Cam, DE, BrDa, CaSi, Sor1, ChHu, GGR, GHa}. 
The aforementioned papers in the Euclidean case, reduced to our eikonal problem, seem to need at least local essential boundedness of $f$ for their well-posedness results except for the papers~\cite{Tou, Sor, Sor1}. In these three papers, the results regarding uniqueness and existence results rely on regularity conditions on the discontinuity set of $f$ or certain controllability assumptions for the optimal control formulation. Such conditions can hardly be extended further to our relaxed geometric setting due to the lack of smooth structure in general metric spaces. 

In the present paper we aim to study in a direct manner the eikonal equation with potentially wilder 
discontinuity than those considered in~\cite{Tou, Sor, Sor1}. We drop the local boundedness of 
the inhomogeneous data $f$ and allow it to be as discontinuous as a function in $L^p(\Omega)$.  It is worth adding that an important tool called $L^p$ viscosity solution theory has been developed for second order uniformly elliptic equations with measurable data in 
\cite{CCKS, CKLS}, but it does not seem to apply to the first order case. For the Dirichlet problem \eqref{eikonal}\eqref{dirichlet}, our primary observation is that the formula \eqref{control formula} actually requires nothing more than the path integrability of $f$ along a curve connecting $x$ to the boundary, which is a much weaker assumption than the continuity or even boundedness of $f$. Let us present a simple 
one-dimensional example to clarify this point. 

Let $\Omega=(-1, 1)\subset \R$ and set $f: \Omega\to \R$ such that
\begin{equation}\label{eq no-bilip1}
f(x)=\frac{1}{\sqrt{|x|}}, \quad x\in [-1, 1]\setminus\{0\}.
\end{equation}
The value $f(0)$ can be chosen to be any positive real number in order to meet the requirement \eqref{f-lower}. Such type of singular eikonal equations finds applications in lens design for wireless communication systems \cite{Chen-NC}. It is clear that $f\in L^p(\Omega)$ with $p\in [1, 2)$. Therefore the formula in \eqref{control formula} still makes sense for given boundary data $g$ at $\pm 1$. For instance, setting $g(\pm1)=0$, we see that the formula immediately yields
\begin{equation}\label{eq ex dis eikonal}
u(x)=2(1-\sqrt{|x|}) \quad \text{for $x\in [-1, 1]$.}
\end{equation}
It seems to be the correct  
solution of \eqref{eikonal}\eqref{dirichlet},  for it satisfies the standard definition of viscosity solutions if we consider $f(0)=\infty$. Also, one can verify the validity of approximation via truncation of the running cost. Indeed, taking $f_M=\min\{f, M\}$ for $M>0$ large, we can obtain a unique viscosity solution for the truncated problem with the bounded function $f_M$. By direct calculations, we easily see that the approximate solution converges uniformly to $u$ in $[-1, 1]$. Moreover, although in general we cannot expect Lipschitz regularity of solutions due to the loss of boundness of $f$, the example shows that the solution enjoys H\"older regularity that depends on the integrability of $f$. This can be viewed as a natural consequence corresponding to the Morrey-Sobolev inequality, as the eikonal equation suggests $u\in W^{1, p}(\Omega)$ at least formally. 

To the best of our knowledge, the well-posedness and regularity issue for this type of discontinuous eikonal equations have not been carefully examined even in the Euclidean case. 
We are therefore motivated to tackle these problems in a general way that permits applications to a broad class of metric measure spaces with length structure. We develop the notion of Monge solutions used in \cite{NeSu, BrDa, LShZ} and show uniqueness and existence of solutions in this general setting. We also establish H\"older regularity of the solution under certain regularity assumptions on $f$ and $\X$. More details about our results will be given momentarily. 

First-order Hamilton-Jacobi equations in metric spaces have recently attracted great attention for applications in optimal transport \cite{AGS13, Vbook}, mean field games \cite{CaNotes}, topological networks \cite{SchC, IMZ,  ACCT,  IMo1, IMo2} etc. Concerning 
the equations that depend on the gradient in terms of its norm, we refer to \cite{AF, GHN, GaS, Na1} for several different notions of viscosity solutions as well as the associated uniqueness and existence results. In our previous paper  \cite{LShZ},  we proved the equivalence of these solutions and introduce another alternative notion, extending the approach of Monge solutions proposed in \cite{NeSu} to the eikonal equation in complete length spaces; see also an application of this Monge-type notion to the study of the first eigenvalue problem for the infinity Laplacian in metric spaces \cite{LMit1}. Our current work further develops the notion of Monge solutions so that viscosity solution theory will be available to handle a class of nonlinear equations in general metric measure spaces with data having a wider range of discontinuities. 

Our work is new even from the point of view of pure analysis.
There has been great progress on various aspects of  first order analysis on metric measure spaces, including first order Sobolev spaces and their relation to variational problems and partial differential equations. One central notion that generalizes the norm of the gradient of a Sobolev function in Euclidean spaces is the so-called upper gradient. We refer the reader to \cite{HKSTBook} for a survey of the Sobolev spaces theory built upon this notion. From this point of view, the Dirichlet problem~\eqref{eikonal} with boundary data~\eqref{dirichlet} can be interpreted as an extension problem for a function $u$ to $\Oba$ with boundary value $g$ assigned on $\partial \Omega$ and the upper gradient $f$ prescribed in $\Omega$. In general, many possible choices can be expected, among which the viscosity solution constructed by the formula \eqref{control formula} is usually considered to be the most ``physical''. This perspective can also be found in the work~\cite{Che}.
There is certainly no need to restrict the prescribed gradient to a continuous or bounded function class. We are thus inspired to address the extension problem for $f\in L^p(\Omega)$ and investigate properties of the extended function $u$.


\subsection{General results on existence and uniqueness of solutions}
Our presentation can be divided into two parts. The first part, consisting of 
Sections \ref{sec:optical}--\ref{sec:existence}, is a general PDE theory with only the length structure and curve-wise integrals of $f$ involved. In our main well-posedness results, we impose certain key conditions on the closure of the domain 
$\overline{\Omega}$ and the inhomogeneous term $f$. 
In the second part, adding a measure structure to the space, we provide more specific sufficient conditions for those assumptions on $\overline{\Omega}$ and $f$.

One of the crucial steps in the first part is to find an appropriate notion of viscosity solutions in this setting that agrees with the representation formula \eqref{control formula} and justifies a comparison principle for us to prove the uniqueness. As mentioned above, we adopt the notion of Monge solutions, which in the case of a continuous $f$, requires
\begin{equation}\label{monge-cont}
|\nabla^- u|(x)=f(x)\quad \text{for all $x\in \Omega$}, 
\end{equation}
where, for a locally Lipschitz function $u$, the subslope $|\nabla^- u|$ is defined by
\[
|\nabla^-u|(x)=\limsup_{d(x, y)\to 0} \frac{\max\{u(x)-u(y), 0\}}{d(x,y)}.
\]
In the Euclidean space, this definition is known to be equivalent to the usual viscosity solutions \cite{NeSu}. One of its advantages is that it avoids using smooth test functions, whose lack of availability is indeed a key difficulty in the study of Hamilton-Jacobi equations in general metric spaces. 

When $f$ is of class $L^p(\Omega)$, the definition in \eqref{monge-cont} may not apply directly. For example, if one changes the value of $f(0)$ to be finite in the aforementioned example, then the expected solution \eqref{eq ex dis eikonal} does not satisfy definition \eqref{monge-cont} at $0$. 
We thus need to somehow exploit collective information of $f$ rather than its pointwise value. Also, we always assume that the lower bound condition \eqref{f-lower} holds for $f$ in order to obtain uniqueness. Heuristically speaking, since $f$ is uniformly positive in $\Omega$, we can rewrite the equation \eqref{eikonal} as 
\[
\frac{|\nabla u|(x)}{f(x)}=1\quad \text{in $\Omega$}
\]
and adopt a new pointwise slope by incorporating the value of
$f(x)$ into the metric. This approach enables us to study the eikonal equation as if 
the right hand side is a constant. 
Such connection can be rigorously realized by introducing the following new metric, which is also called optical length function in \cite{NeSu, BrDa} etc.,  
\begin{equation}\label{ol-fun}
L_f(x, y):=\inf\left\{\int_\gamma f\, ds: \ \text{$\gamma\in \Gamma_f(x, y)$}\right\}\quad \text{for any $x, y\in \Oba$,}
\end{equation}
where $\Gamma_f(x, y)$ denotes the collection of all curves $\gamma$ connecting $x$ and $y$ in $\overline{\Omega}$ on which the path integral is finite, namely, 
\[
\Gamma_f(x, y)=\left\{\gamma:[0, \ell]\to \Oba: \int_\gamma f\, ds<\infty,\ \gamma(0)=x, \gamma(\ell)=y \text{ and } |\gamma'|(s)=1 \text{ for a.e. }s\right\}.
\]
We say that the curves $\gamma$ in $\Gamma_f(x, y)$ are admissible with respect to $f$. 
We shall assume throughout Sections \ref{sec:optical}--\ref{sec:existence} that
\begin{equation}\label{c-integrable}
\Gamma_{f}(x, y)\neq \emptyset \quad \text{for all $x, y\in \Oba$}
\end{equation}
so that the formula in \eqref{control formula} is well defined for any $x\in \Oba$. In Section \ref{sec:admissible}, we discuss some sufficient assumptions on $\overline{\Omega}$ implying \eqref{c-integrable}.

Adopting the subslope with respect to this newly defined metric $L_f$, defined by
\begin{equation}\label{new subslope}
|\nabla_f^- u|(x):=\limsup_{L_f(x, y)\to 0} \frac{\max\{u(x)-u(y), 0\}}{L_f(x,y)}\quad \text{for $x\in \Omega$,}
\end{equation}
one can simply define the Monge solution of \eqref{eikonal} in Definition \ref{Monge def} by asking that $u$ satisfies
\begin{equation}\label{new monge}
|\nabla_f^- u|(x)=1\quad \text{for all $x\in \Omega$.}
\end{equation}
The same idea of changing metrics was also outlined in \cite{Bu} in connection with metric geometry and in~\cite{TrYu} 
for applications in homogenization.  

This new definition of Monge solutions is consistent with the well-studied case where 
$f$ is continuous. In fact, if $f$ is continuous at $x_0\in \Omega$, then 
\begin{equation}\label{length cont}
\frac{L_f(x, x_0)}{d(x, x_0)}\to f(x_0) \quad \text{as $d(x, x_0)\to 0$.}
\end{equation}
It is then easily seen that $|\nabla_f^-u|(x_0)=1$ and $|\nabla^- u|(x_0)=f(x_0)$ are equivalent. 

Switching the metric from $d$ to $L_f$ enables us to reduce the possibly discontinuous inhomogeneous data in~\eqref{monge-cont} to the continuous standard form~\eqref{new monge}. As a consequence, most of our arguments in \cite{LShZ} can be applied to establish the comparison principle and prove existence of solutions of \eqref{eikonal},\,\eqref{dirichlet} under the compatibility condition 
\begin{equation}\label{growth}
g(x)\le  L_f(x, y)+g(y)\quad \text{for all $x, y\in \partial \Omega$.}
\end{equation} 
However, a major challenge concerning the topological change of the space appears in the proof of comparison principle. Note that the metrics $L_f$ and $d$ are not topologically equivalent in general, especially when $f$ is unbounded. One can find examples where $d(x, y)\to 0$ fails to imply $L_f(x, y)\to 0$; see Example \ref{ex singular}. For most of our analysis, 
we 
impose the following key compatibility condition 
\begin{equation}\label{ol-uc}
\sup\{L_f(x, y): x, y\in \Oba,\ d(x, y)\leq r\}\to 0 \quad \text{as $r\to 0$},
\end{equation}
which ensures that the topology induced by $L_f$ is in agreement with the metric topology inherited from $(\X,d)$, and
enables us to obtain not only the uniqueness but also the uniform continuity of Monge solutions. 

In general, if  \eqref{ol-uc} fails to hold, then, because of the change in the metric, 
the domain $\Omega$ that is bounded in $d$ may become unbounded in $L_f$. The notions of interior and boundary points of $\Omega$ may also change accordingly. 
However, it is still possible to show uniqueness and existence of Monge solutions if we have boundedness of  $\Oba$ with respect to $L_f$ as well as a weaker version of \eqref{ol-uc} as below:
\begin{equation}\label{ol-uc2}
\sup\{L_f(x, y): x, y\in \Oba\setminus \Omega_{r},\ d(x, y)\leq r\}\to 0 \quad \text{as $r\to 0$},
\end{equation}
where $\Omega_{r}=\{x\in \Omega: d(x,\partial\Omega) >r\}$. Here, we denote $d(x, \partial \Omega)=\inf_{y\in \partial \Omega} d(x, y)$. In this case, we can get a unique Monge solution that is Lipschitz continuous in $L_f$ but possibly discontinuous in $d$; see Example \ref{ex disc sol} for a concrete example. Note that the completeness of $\Oba$ is preserved under our metric change (Lemma \ref{lem complete}), which enables us to establish a comparison principle based on Ekeland's variational principle.


Our main result can be summarized as follows. Its proof is given in Section~\ref{sec:existence}.

\begin{thm}[Existence and uniqueness of solutions]\label{exist}
Let $(\X, d)$ be a complete length space and $\Omega\subsetneq\X$ be a bounded domain. 
Assume that \eqref{f-lower} and \eqref{c-integrable} hold. Let $L_f$ be given by \eqref{ol-fun}, and 
$g: \partial \Omega\to \R$ be bounded.
Consider the function $u$ defined by
\begin{equation}\label{lax}
u(x):=\inf_{y\in \partial \Omega} \{L_f(x, y)+g(y)\}, \quad x\in \overline{\Omega}.
\end{equation}

 Then the following results hold.  
\begin{enumerate}
\item[(i)] $u$ is Lipschitz in $\Oba$ with respect to $L_f$ and is a Monge solution of \eqref{eikonal}. 
\item[(ii)] If $g$ further satisfies the compatibility condition  \eqref{growth}, $(\Oba, L_f)$ is bounded, and \eqref{ol-uc2} holds,
then $u$ is the unique Monge solution 
of the  Dirichlet problem \eqref{eikonal},\,\eqref{dirichlet} with~\eqref{dirichlet} interpreted as
\begin{equation}\label{bdry regularity}
\sup\{|u(x)-g(y)|: x\in \Oba,\ y\in \partial \Omega,\ d(x, y)\leq \delta\}\to 0\quad \text{as $\delta\to 0$}.
\end{equation}
\item[(iii)] If $g$ satisfies  \eqref{growth}, and \eqref{ol-uc} holds, then $u$ is the unique Monge solution,  uniformly continuous with respect to $d$, of the Dirichlet problem \eqref{eikonal},\,\eqref{dirichlet} with~\eqref{dirichlet} interpreted as \eqref{bdry regularity}.
\end{enumerate}
\end{thm}
The interpretation \eqref{bdry regularity} is automatically guaranteed by~\eqref{dirichlet} if $(\Oba, L_f)$ is assumed to be compact. 
Our comparison results with and without compactness of $(\Oba, L_f)$ are presented in Section \ref{sec:comparison-proof}. 
We also include a discussion in Section \ref{sec:bdry-loss} about the case when the compatibility condition \eqref{growth} on $g$ fails to hold. In the Euclidean space, this corresponds to loss of Dirichlet boundary data, and one needs to relax the boundary condition in the viscosity sense; see for example \cite{So1, So2, I7} for the well-posedness results under certain regularity conditions on $\partial \Omega$. Thanks to the general setting considered in this
present paper, we treat this problem in a more direct way. If the boundary data $g$ on a part of $\partial \Omega$ is lost, we take the remaining piece, denoted by $\Sigma_g$, as the real boundary and regard $\Oba\setminus \Sigma_g$ as the interior of the domain. In fact, the formula \eqref{control formula} does yield the Monge solution property of $u$ in $\Oba\setminus \Sigma_g$. Such flexibility in changing the notion of boundary and interior points leads us to uniqueness and existence of solutions to the reduced Dirichlet problem. 

\subsection{Regularity of solutions and further discussions}
The second half of our presentation is devoted to understanding more deeply the key assumptions \eqref{c-integrable} and \eqref{ol-uc} in Theorem \ref{exist}. We study this regularity problem in the setting of a general metric measure space $(\X, d, \mu)$ with the measure $\mu$ assumed to be doubling. We recall that a locally finite Borel measure $\mu$ is doubling if there exists $C_d\ge 1$ such that $\mu(B_{2r}(x))\le C_d\mu(B_r(x))$ for all $r>0$, where $B_r(x)$ denotes the open metric ball centered at $x\in \X$ with radius $r>0$, i.e., 
\[
B_r(x):=\{y\in \X\, :\, d(x,y)<r\}.
\] 
We also call $Q=\log_2 C_d$ the homogeneous dimension of the doubling metric measure space. 
It is not difficult to see that by the doubling property of $\mu$, we have, for any metric balls $B_r(x)\subset B_R(x)\subset \X$,
\begin{equation}\label{hdim}
\frac{\mu (B_r(x))}{\mu(B_R(x))}\ge C{\left(\frac{r}{R}\right)}^Q.
\end{equation}
We prove the property \eqref{ol-uc} and H\"{o}lder regularity with respect to the metric $d$ of the solution $u$ defined by \eqref{lax} under either of the following two assumptions. 
\begin{itemize}
\item[(A1)] $\Oba$ satisfies a $p$-Poincar\'e inequality for some finite $p>\max\{1, Q\}$ and $f\in L^p(\Oba)$, where $Q>0$ is the homogeneous dimension  from~\eqref{hdim} above. 
\item[(A2)] $\Oba$ satisfies an $\infty$-Poincar\'e inequality and $f\in L^\infty(\Oba)$. 
\end{itemize}

\begin{thm}[Regularity of solutions]\label{thm holder}
Let $(\Oba, d, \mu)$ be a  complete bounded metric measure space with $\mu$ a doubling measure. Let $g: \partial \Omega\to \R$ be bounded and $u$ be the Monge solution defined by \eqref{lax}. 
If $\Oba$ and $f$ satisfy (A1), then $u$ is $(1-\frac{Q}{p})$-H\"older continuous in $\Oba$ with respect to the metric $d$. If $\Oba$ and $f$ satisfy (A2), then $u$ is Lipschtiz continuous in $\Oba$ with respect to the metric $d$. 
\end{thm} 
Here, the $p$-Poincar\'e inequality is a condition that indicates the richness of curves of a space; see Section \ref{sec:admissible} for a more precise description and \cite{HKSTBook} for more detailed introduction. In our assumptions, a balance is taken between the regularity of the space $\Oba$ and the function $f$. In order to consider merely $p$-integrable $f$ with a finite $p$, we assume $p$-Poincar\'e inequality in (A1), which is  stronger than the $\infty$-Poincar\'e inequality in (A2),

Our results in Theorem \ref{thm holder} resemble the Morrey-Sobolev embedding theorem.  In the proof, we apply a geometric characterization of the Poincar\'e inequality provided by \cite{DJS1} 
for both cases of (A1) and (A2) with adaptations for our eikonal equation. We can also show that $u$ is in the Sobolev class $N^{1, p}(\Oba)$ under the conditions of Theorem \ref{thm holder}.  While (A1) and (A2) seem close to optimal for us to obtain \eqref{ol-uc}, both of them are actually too strong to directly show \eqref{c-integrable}, which is far weaker than \eqref{ol-uc} and simply requires the existence of one curve on which $f$ is integrable. We construct an example in $\R^2$, showing that \eqref{c-integrable} can be obtained even for some $f\notin L^2(\Omega)$; see Example~\ref{ex c-integrable}. It would be interesting to find more general sufficient conditions for \eqref{c-integrable}. 

It is worth remarking that the $L^p$ class in our study is only used to describe the extent of discontinuity of $f$ and cannot be understood as the usual $p$-integrable function space. Since the Monge solution obtained by \eqref{control formula} substantially depends on path integrals of $f$, changing, especially decreasing, the value of $f$ even on a null set will drastically affect the solution. In other words, our Monge solution strongly depends on 
the choice of representatives of the equivalence class for $f\in L^p(\Oba)$. 

Section \ref{sec:lp} is devoted to discussions on how to alleviate the instability of solutions with respect to $f$ increasing on null sets. Our method generalizes the approach in \cite{CaSi1, BrDa}. Instead of taking the solution for $L^p$ function $f$, we generate a class of solutions for different representatives from its equivalence class under the assumption (A1) or (A2). More precisely, for any null set $N\subset \Oba$, we can find a Monge solution of \eqref{eikonal}\eqref{dirichlet} with $f$ replaced by $f_N=f+\infty \chi_N$, where $\chi_N$ denotes the characteristic function of $N$. This change in $f$ amounts to restricting $f$ to be integrable only on curves transversal to $N$. Here, a curve is said to be transversal to $N$ if $\mathcal{H}^1(\gamma^{-1}(N))=0$, where $\mathcal{H}^1$ stands for the one-dimensional Hausdorff measure. By substituting $\Gamma_f(x, y)$ for each pair of points $x, y\in \Oba$ with the class of transversal curves
\begin{equation}\label{cf}
\begin{aligned}
{\Gamma}^{N}_f(x, y)=\bigg\{\gamma: [0,\ell]\to \Oba\ :  &\  \int_\gamma f\, ds<\infty,\, \gamma(0)=x, \gamma(\ell)=y\ \text{ and }\\
&  |\gamma'|(s)=1 \text{ for a.e. }s, \ \text{with}\ \mathcal{H}^1(\gamma^{-1}(N))=0 \bigg\}
\end{aligned}
\end{equation}
and taking the associated optical length function 
\begin{equation}\label{olfN}
L^N_f(x, y)=\inf  \left\{\int_\gamma f\, ds\,:\, \gamma\in {\Gamma}_f^{N}(x,y)\right\},
\end{equation}
we can obtain the same existence and uniqueness results for any given null set $N$. 

We are particularly interested in the maximal solution over all null sets $N$, as it corresponds to the choice of the largest function in the equivalence class of $f$ and represents the strongest possible instability. It can be defined by 
\begin{equation}\label{lax2}
\tilde{u}(x)=\inf_{y\in \partial \Omega} \{\tilde{L}_f(x, y)+g(y)\}, \quad x\in \overline{\Omega},
\end{equation}
where $\tilde{L}_f$ is the maximal optical length function over all $N$, that is, 
\begin{equation}\label{olf}
\tilde{L}_f(x,y)=\sup_{\mu(N)=0} L_f^N(x, y)=\sup_{\mu(N)=0}\inf  \left\{\int_\gamma f\, ds\,:\, \gamma\in {\Gamma}_f^{N}(x,y)\right\}.
\end{equation}
We proved in Proposition \ref{maxLip} and Remark \ref{rmk weak sol} that $\tilde{u}$ is the maximal weak solution of \eqref{eikonal},\,\eqref{dirichlet} under 
(A2) together with appropriate assumptions on $\Oba$ and boundary data $g$ as well as almost everywhere continuity of $f$. Here, the notion of weak solutions is a generalization of the Euclidean counterpart, requiring $\tilde{u}$ to satisfy $|\nabla \tilde{u}|=f$ almost everywhere in $\Omega$, where $|\nabla \tilde{u}|$ denotes the slope of $u$, i.e., for $x\in \Omega$,
\[
|\nabla \tilde{u}|(x):=\limsup_{d(x, y)\to 0} \frac{|\tilde{u}(y)-\tilde{u}(x)|}{d(x,y)}.
\] 
This result extends the discussion about the well-known maximal weak solution characterization of viscosity solutions in the Euclidean space \cite{NeSu, BrDa} to the setting of metric measure spaces and possibly unbounded discontinuous inhomogeneous term $f$. It is however not clear to us how to handle the case when $f$ is not continuous almost everywhere. 
Another important open question is under what assumptions can we can eliminate the instability, i.e., $u=\tilde{u}$ in $\Oba$. Some discussions regarding this issue can be found in~\cite{Sor} for the Euclidean space but the case for general metric spaces seems more difficult. 

Let us conclude the introduction by mentioning that in this work we choose not to pursue more general Hamilton-Jacobi equations.  One can easily generalize our method if a control-based formula is available and a new metric incorporating the discontinuity of the Hamiltonian can be found.  Even for the eikonal equation itself, our general setup actually implicitly allows more general dependence of the Hamiltonian on the space variable. Typical examples of $\X$ include the sub-Riemannian manifolds. For instance, when $(\X, d)$ is taken to be the first Heisenberg group with the Carnot-Carath\'eodory metric, the eikonal equation \eqref{eikonal} can be written in the Euclidean coordinates as
\[
\sqrt{\left(u_x-{y\over 2}u_z\right )^2+\left(u_y+{x\over 2}u_z\right)^2}=f(x, y, z) \quad \text{for $(x, y, z)\in \R^3$.}
\]
We refer to \cite{Bi4, Alb2, ACS} etc. for discussions on the eikonal equation with $f\equiv 1$ in the sub-Riemannian setting.

\subsection*{Acknowledgements}
The authors thank Professors Yoshikazu Giga, Nao Hamamuki and Atsushi Nakayasu for valuable comments on the first draft of the paper.
The work of the first author was supported by JSPS Grant-in-Aid for Scientific 
Research (No.~19K03574, No.~22K03396).  
The work of the second author is partially supported by
the grant DMS-\#2054960 of the National Science Foundation (U.S.A.). 
The work of the third author was supported by JSPS Grant-in-Aid for Research Activity Start-up (No.~20K22315) and JSPS Grant-in-Aid for Early-Career Scientists (No.~22K13947).

\section{Optical length function}\label{sec:optical}

The notion of optical length dates back to the paper of Houstoun~\cite{Hous}, see also~\cite[(3)]{SeKe} in the context
of the behavior of light rays in general relativity. Briani and Davini \cite{BrDa} use this notion to define the Monge solution for discontinuous Hamiltonians in the Euclidean space. See recent developments of this approach on Carnot groups in \cite{EGV}. In the context of metric measure spaces, the notion of optical length is 
constructed in~\cite{Che}, but there the terminology of optical length was not used. Other applications can be found in \cite{HKSTBook, JJRRS}.

We define the optical length function  $L_f: \Oba\times \Oba\to \R$ as in \eqref{ol-fun}. Hereafter,  we sometimes adopt the notation $I_f(\gamma):=\int_\gamma f\, ds$ for a rectifiable curve $\gamma$.
Under our standing assumption~\eqref{c-integrable}, 
it is not difficult to see that $L_f$ satisfies all the axioms of a metric on $\Oba$. 
\begin{lem}[Metric properties of optical length]\label{lem metric-like}
Let $(\X, d)$ be a complete length space and $\Omega\subsetneq\X$ be a bounded domain. Assume that $f$ satisfies \eqref{f-lower}.  Let $L_f$ be the optical length function defined by \eqref{ol-fun}. Assume in addition that \eqref{c-integrable} holds. Then $L_f$ satisfies the following: 
\begin{enumerate}
\item[(1)] $L_f(x, y)\geq 0$ for all $x, y\in \Oba$, and $L_f(x, y)=0$ holds if and only if $x=y$;
\item[(2)] $L_f(x, y)=L_f(y, x)$ for all $x, y\in \Oba$;
\item[(3)] $L_f(x, y)\leq L_f(x, z)+L_f(y, z)$ for all $x, y, z\in \Oba$.
\end{enumerate}
Moreover, 
\begin{equation}\label{lower lip}
L_f(x, y)\geq  \alpha d(x, y)\quad  \text{for all $x, y\in \Omega$.} 
\end{equation}
\end{lem}
\begin{proof}
By \eqref{f-lower}, it is clear that $L_f\geq 0$ in $\Oba\times \Oba$ and in addition, $L_f(x,y)=0$ if and only if $x=y$. 
Moreover, the definition \eqref{ol-fun} immediately implies \eqref{lower lip}.
This completes the proof of (1). The statement (2) is obvious. We finally prove (3), which is also quite straightforward. 
For any $\vep>0$, there exist $\gamma_1\in \Gamma_f(x, z)$ and $\gamma_2\in \Gamma_f(z, y)$ such that 
\[
L_f(x, z)\geq I_f(\gamma_1)-\vep, \quad L_f(z, y)\geq I_f(\gamma_2)-\vep. 
\]
By connecting the curve $\gamma_1$ and $\gamma_2$ to build a curve joining $x$ and $y$, we can easily see that 
\[
L_f(x, y)\leq I_f(\gamma_1)+I_f(\gamma_2)\leq L_f(x, z)+L_f(z, y)+2\vep.
\]
We conclude the proof by sending $\vep\to 0$.
\end{proof} 

\begin{rmk}
Under the assumption \eqref{c-integrable}, using the metric $L_f$, 
we can introduce a new notion of length of curves in $\Oba$. For any curve $\gamma: [a, b]\to \Oba$, define
\[
\ell_f(\gamma)=\sup_{a=t_0<t_1<\cdots<t_k=b} \sum_{j=0}^{k-1}L_f(\gamma(t_j),\gamma(t_{j+1})).
\]
We note that when $f=1$ in $\Oba$, $L_f$ defines the intrinsic metric and $\Oba$ equipped with this metric is a length space. 
For a general function $f$ one can still verify that $\Oba$ is a length space with metric $L_f$ and the length 
of curves in this metric is given by $\ell_f$ 
defined above. 
\end{rmk}


As an immediate consequence of \eqref{lower lip}, the closure/boundary of $\Omega$ with respect to $L_f$ is contained in the closure/boundary with respect to $d$.
We can obtain the bi-Lipschitz equivalence between $d$ and $L_f$ if a reverse version of \eqref{lower lip} holds. It is the case when $f$ is uniformly bounded. However we cannot expect the bi-Lipschitz equivalence when $f$ is unbounded. A simple example is as follows. 

\begin{example}\label{ex no-bilip}
Let $\X=\R$ with $d$ the Euclidean metric and $\mu$ the one-dimensional Lebesgue measure. Let $\Omega=(-1, 1)$ 
and $f$ be the function given by \eqref{eq no-bilip1}. In this case, we have
\[
L_f(x, 0)=\int_0^x f(s)\, ds=2\sqrt{|x|}
\]
for all $x\in (-1, 1)$ and therefore
\[
\frac{L_f(x, 0)}{d(x, 0)}=\frac{2}{\sqrt{|x|}}\to \infty\quad \text{as $x\to 0$.}
\]
Note however that $(-1,1)$ is bounded with respect to the metric $L_f$, and the topology generated by $L_f$ 
agrees with the Euclidean topology on $(-1,1)$; moreover, $[-1,1]$ is complete and compact in both metrics.
\end{example}

One can further consider the case when $L_f$ is unbounded in $\Oba$; in fact, in Example \ref{ex no-bilip}, 
replacing $f$ with the following function 
\[
f(x)=\frac{1}{|x|},  \quad x\in (-1, 1)\setminus\{0\},
\]
we see that $L_f(0, x)=\infty$ for any $x\in (-1, 1)\setminus\{0\}$. Observe that this choice of $f$ does not belong to $L^p((-1,1))$ for any $p\ge 1$.
Since such situation is not our focus in this work, we do not pursue this direction here and leave it for future discussions.

Even if $L_f$ is bounded in $\Oba\times \Oba$, due to the unboundedness of $f$ in general it may happen that 
$L_f(x, y)\not\to 0$ as $d(x, y)\to 0$. 
An example on metric graphs is as follows. 

\begin{example}\label{ex singular}
In $\R^2$, let $e_0$ be the closed line segment between points $(0, 0)$ and $(1, 0)$, and $e_j$ be the line 
segments connecting $P_j=(1/j, 0)$ and $Q_j=(1/j, 1/j)$ for $j=1, 2, \ldots$ We take 
$\X=\Oba=\bigcup_{j=0, 1, \ldots} e_j$ with $d$ being the intrinsic metric of this graph. See Figure~\ref{fig:comb}.
\begin{figure}[H]
  \centering
  \includegraphics[height=5.2cm]{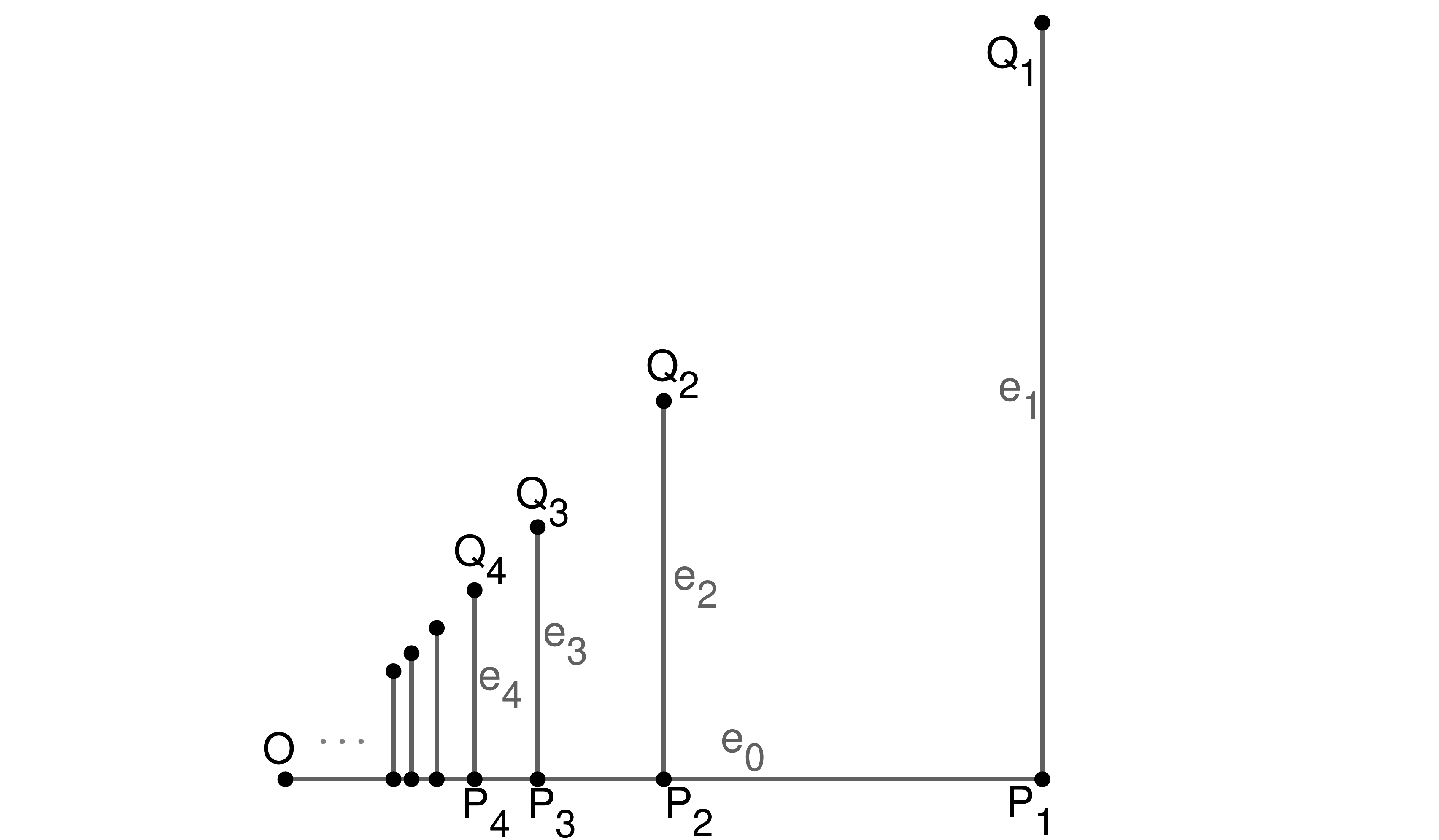}
  \caption{Example \ref{ex singular}}
  \label{fig:comb}
 \end{figure}
It is not difficult to see that $(\X, d)$ is a complete geodesics space. For $x\in \X$ with coordinates $(x_1, x_2)$ in $\R^2$, let 
\[
f(x_1, x_2)=\begin{dcases}
1/x_1 &\text{if $x_2> 0$;}\\
1 & \text{if $x_2=0$.}
\end{dcases}
\]
Denoting $O=(0, 0)$, by direct calculations, for any $j\geq 2$ we have
\[
L_f(O, Q_j)=j \int_0^{1\over j}  \, ds+\int_0^{1\over j} ds=1+{1\over j}. 
\]
(Here we choose the optimal integration path from $Q_j$ to $P_j$ and then to $O$.) 
We thus observe that as $j\to \infty$, $d(O, Q_j)\to 0$ but $L_f(O, Q_j)\to 1$. 

\end{example}

\begin{lem}[Completeness]\label{lem complete}
Let $(\X, d)$ be a complete length space and $\Omega\subsetneq\X$ be an open set. Assume that $f$ satisfies \eqref{f-lower}.  Let $L_f$ be the optical length function defined by \eqref{ol-fun}. Then $(\Oba, L_f)$ is complete. 
\end{lem}
\begin{proof}
Suppose that $x_j \in \Oba$ is a Cauchy sequence with respect to the metric $L_f$. Thanks to \eqref{lower lip} in Lemma \ref{lem metric-like}, we see that $\{x_j\}$ is also a Cauchy sequence with respect to the metric $d$. Since $\Oba$ is a closed set in the complete space $\X$, there exists $x_0\in \Oba$ such that $d(x_j, x_0)\to 0$ as $j\to \infty$.  \\
By definition of Cauchy sequences, for any $\vep>0$, we can take $x_{j_1}$ such that $L_f(x_{j_1}, x_i)<\vep/2$ for all $i\geq j_1$. Similarly, we can choose $x_{j_2}$ with $j_2>j_1$ satisfying $L_f(x_{j_2}, x_i)<\vep/4$ for all $i\geq j_2$. We repeat this process to obtain a sequence $x_{j_k}$ such that $L_f(x_{j_k}, x_{j_{k+1}})<2^{-k}\vep$. By definition of $L_f$, we can find curves $\gamma_k$ in $\Oba$ joining $x_{j_k}$ and $x_{j_k+1}$ such that 
\[
\int_{\gamma_k} f\, ds< 2^{-k}\vep. 
\]
Concatenating these curves in order, we build a curve $\gamma$ connecting $x_{j_1}$ to $x_0$ satisfying 
\[
\int_\gamma f\, ds\leq \sum_{k\geq 1} 2^{-k}\vep=\vep,
\]
which yields $L_f(x_{j_1}, x_0)\leq \vep$. In view of the arbitrariness of $\vep>0$, we have actually proved the convergence of $x_j$ to $x_0$ in the metric $L_f$. 
\end{proof}

We cannot guarantee the compactness of $(\overline{\Omega}, L_f)$ without imposing further assumptions. 
Indeed, in Example \ref{ex singular} we see that the sequence $\{Q_j\}_{j\geq 2}$ is bounded but without a convergent subsequence in $(\overline{\Omega},L_f)$; indeed, for each $j\ne k$ we have that 
$L_f(Q_j, Q_k)\ge 2$. It follows that $(\overline{\Omega},L_f)$ is not compact.

\section{Monge Solutions and Comparison Principle}\label{sec:comparison}

\subsection{Definition of Monge solutions}

Let us now study the Dirichlet problem \eqref{eikonal},\,\eqref{dirichlet}. We begin with the definition of Monge solutions to \eqref{eikonal}. 

\begin{defi}[Monge solutions] \label{Monge def}
We say that a locally bounded  function $u: \Omega\to \R$  is a Monge solution (resp. subsolution, supersolution) to 
\eqref{eikonal} in $\Omega$ if for every $x_0\in \Omega$,
\begin{equation}\label{monge-sub}
\limsup_{L_f(x, x_0)\to 0}\frac{u(x_0)-u(x)}{L_f(x_0, x)}= 1 \quad (\text{resp., } \leq, \geq).
\end{equation}
\end{defi}

We stress that the topology in the local boundedness of $u$ and limsup in the definition above is taken with respect to the metric $L_f$. Since $L_f(x, x_0)\to 0$ implies $d(x, x_0)\to 0$ by \eqref{lower lip}, we see that 
\begin{equation}\label{monge-sol}
\limsup_{L_f(x, x_0)\to 0}\frac{u(x_0)-u(x)}{L_f(x_0, x)}= 1\quad \text{for every $x_0\in \Omega$}
\end{equation}
 implies 
\begin{equation}\label{monge d-metric1}
\limsup_{d(x, x_0)\to 0}\frac{u(x_0)-u(x)}{L_f(x_0, x)}\geq 1\quad\text{for every $x_0\in \Omega$}.
\end{equation}
In general, we cannot expect that the reverse inequality of \eqref{monge d-metric1} holds without further assumptions like \eqref{ol-uc}. 


We also emphasize that our definition does not require any continuity or semicontinuity of $u$. 
This relaxation constitutes a difference from the standard definition of discontinuous viscosity solutions. Recall that when defining a (possibly discontinuous) subsolution $u$, we usually assume $u\in USC(\Omega)$ (and $u\in LSC(\Omega)$ for the symmetric supersolution definition). Here, $USC(\Omega)$ and $LSC(\Omega)$ denote the classes of upper and lower semicontinuous functions in $\Omega$ respectively with respect to $d$. 
We do not adopt such restrictions here. 

\begin{rmk}\label{rmk cont1}
Our definition above is
 a direct generalization of the standard notion of  
Monge solutions for $f\in C(\Omega)$. Recall that $\X$ is assumed to be a length space.
When $f$ is continuous at $x_0$, we can show that~\eqref{length cont} holds
and that $L_f(x,x_0)\to 0$ as $d(x,x_0)\to 0$. Our generalized notion is consistent with the case for continuous $f$ discussed in \cite{LShZ}.
\end{rmk}

Using the subslope defined in \eqref{new subslope}, we can see in a straightforward manner that our Monge 
solutions (resp., subsolutions, supersolutions) of~\eqref{eikonal} can simply be understood as Monge solutions (resp., subsolutions, supersolutions) of 
\begin{equation}\label{eikonal new}
|\nabla_f u|=1 \quad \text{in $\Omega$}.
\end{equation}




\begin{prop}
Let $(\X, d)$ be a complete length space and $\Omega\subsetneq\X$ be a bounded domain. 
Assume that \eqref{f-lower} hold. If $u$ is a Monge 
subsolution of \eqref{eikonal}, then for any $x_0\in \Omega$ and any $C>1$,  
\[
u(x)-u(x_0)\geq -CL_f(x, x_0)
\]
for any $x\in \Omega$ sufficiently close to $x_0$. In addition, if \eqref{ol-uc} holds, then $u$ is lower semicontinuous in $\Omega$. 
\end{prop}
On the other hand, Definition \ref{Monge def}, together with \eqref{ol-uc}, does not guarantee the upper semicontinuity of a Monge solution. However, for Monge solutions that arise from the associated Dirichlet problem, we will later use its optimal control interpretation to show the continuity of solutions.

Thanks to the metric change, 
one can essentially apply the results in \cite{LShZ} to \eqref{eikonal new} in the complete
metric space $(\Oba, L_f)$ to show both uniqueness and existence of Monge solutions of \eqref{eikonal}\eqref{dirichlet} under appropriate assumptions on the boundary data $g$. 
The compatibility condition \eqref{growth} we will impose on $g$, combined with \eqref{ol-uc}, basically ensures that the topology stays equivalent when converting the metric from $d$ to $L_f$.

\subsection{Comparison principle}\label{sec:comparison-proof}
Let us prove a comparison principle, where we assume the semicontinuity of Monge sub- and supersolutions with respect to the metric $L_f$. We say that $u: \Oba\to \R$ is upper (resp., lower) semicontinuous with respect to $L_f$, denoted by $u\in USC_{L}(\Oba)$ (resp., $u\in LSC_{L}(\Oba)$), if for every fixed $x\in \Oba$, we have
\[
\limsup_{L_f(x, y)\to 0,\ y\in \Oba}\ u(y)\leq u(x) \quad \left(\text{resp.,} \liminf_{L_f(x, y)\to 0,\ y\in \Oba}\ u(y)\geq u(x)\right).
\]
It is not difficult to see that $u\in USC_{L}(\Oba)$ (resp., $u\in LSC_{L}(\Oba)$) if $u\in USC(\Oba)$ (resp., $u\in LSC(\Oba)$) with respect to $d$, since $L_f(x, y)\to 0$ implies that $d(x, y)\to 0$.


\begin{thm}[Comparison principle for Monge solutions]\label{thm comparison monge}
Let $(\X, d)$ be a complete length space and $\Omega\subsetneq\X$ be a bounded domain. Assume 
that $f$ satisfies \eqref{f-lower}. 
Let $u \in USC_{L}(\Oba)$ and $v \in LSC_{L}(\Oba)$ be respectively a bounded Monge subsolution and a bounded Monge supersolution of \eqref{eikonal}. If 
\begin{equation}\label{bdry verify monge}
\lim_{\delta\to 0}\sup\left\{u(x)-v(x): x\in \Oba, \ d(x, \pO)\leq \delta \right\}\leq 0, 
\end{equation}
then $u\le v$ holds in $\Oba$.
\end{thm}
\begin{proof}
Since $u$ and $v$ are bounded, we may assume that $u, v\geq 0$ by adding a 
positive constant to them. It suffices to show that $\lambda u\le v$ in $\Omega$ for all $\lambda\in (0,1)$. 
Assume by contradiction that 
there exists $\lambda\in (0,1)$ such that $\sup_{\Omega}(\lambda u-v)> 2\tau$ for some $\tau>0$.  
By \eqref{bdry verify monge}, we may take $\delta>0$ small such that 
\[
\lambda u(x)-v(x)\leq u(x)-v(x)\leq \tau
\]
for all $x\in \Oba\setminus \Omega_\delta$, where we denote $\Omega_r=\{x\in \Omega: d(x, \pO)>r\}$ for $r>0$. 
We choose $\vep\in (0, \delta\alpha/2)$ small, where $\alpha>0$ is the lower bound of $f$ as in \eqref{f-lower}, such that
\[
\sup_\Omega (\lambda u-v)>2\tau+\vep^2
\]
and
\begin{equation}\label{eps small}
\vep<1-\lambda.
\end{equation}
Thus there exists $x_0\in \Omega$ such that $\lambda u(x_0)-v(x_0)\geq \sup_\Omega (\lambda u-v)-\vep^2>2\tau$ 
and therefore $x_0\in \Omega_\delta$. 

In view of Lemma \ref{lem complete}, we see that $(\Oba, L_f)$ is complete. 
Since $\lambda u-v$ are bounded from above and upper semicontinuous in $\overline{\Omega}$ with respect to the metric $L_f$, by Ekeland's variational principle (cf. \cite[Theorem 1.1]{Ek1}, \cite[Theorem 1]{Ek2}), 
there exists 
$x_\vep\in \Omega$ such that 
\begin{equation}\label{ekeland1}
L_f(x_\vep, x_0)<\vep,
\end{equation}
\[
\lambda u(x_\vep)-v(x_\vep)\geq \lambda u(x_0)-v(x_0), 
\]
and 
\begin{equation}\label{ekeland2}
\lambda u(x)-v(x)-\vep L_f(x_\vep, x)\leq \lambda u(x_\vep)-v(x_\vep) \quad \text{for all $x\in \Omega$.}
\end{equation}
Note that by \eqref{lower lip}, the relation \eqref{ekeland1} 
combined with the choice of $\vep<\delta\alpha/2$
implies that $x_\vep\in B_{\vep/\alpha}(x_0)\subset \Omega_{\delta/2}$.  
Then from~\eqref{ekeland2}, it follows that 
\begin{equation}\label{comparison monge1}
v(x_\vep)-v(x)\le \lambda u(x_\vep)-\lambda u(x)+\vep L_f(x_\vep, x)\quad \text{for all $x\in \Omega$}
\end{equation}
when $\vep>0$ is small enough. 
Hence, by \eqref{lower lip} we get
\[
\limsup_{L_f(x, x_\vep)\to 0}\frac{v(x_\vep)-v(x)}{L_f(x_\vep, x)}\le \limsup_{L_f(x, x_\vep)\to 0}\frac{\lambda(u(x_\vep)-u(x))}{L_f(x_\vep, x)}+\vep.  
\]
By the Monge subsolution property of $u$ and the Monge supersolution property of $v$, it follows that $1\leq \lambda+\vep$,
which is a contradiction to \eqref{eps small}. Hence, we obtain $\lambda u\le v$ in $\Omega$ for all $0<\lambda<1$. Letting $\lambda\to 1$, we end up with $u\le v$ in $\Omega$. 
Our proof is thus complete.
\end{proof}

\begin{rmk}\label{rem comparison}
If in Theorem \ref{thm comparison monge} we further assume that $(\Oba, L_f)$ is compact, 
then the proof becomes simpler. 
Under this assumption, there is no need to use Ekeland's variational principle, since $u-v$ is upper semicontinuous 
and hence attains a maximum at a point $\hat{x}$ in the compact set $\Oba$. The condition \eqref{bdry verify monge} can also be simplified; it is sufficient to assume that $u\leq v$ on $\partial \Omega$, since it guarantees a positive maximum of $\lambda u-v$ in $\Omega$ for $0<\lambda<1$ in our proof above. We summarize this observation in the following theorem.
\end{rmk}

\begin{thm}[Comparison principle in a compact space]\label{thm comparison geodesic}
Let $(\X, d)$ be a complete length 
 space and $\Omega\subsetneq\X$ be a bounded domain. Assume that $f$ satisfies \eqref{f-lower} and that $(\Oba, L_f)$ is compact. 
Let $u \in USC_{L}(\Oba)$ and $v \in LSC_{L}(\Oba)$ be respectively a bounded Monge subsolution and a bounded Monge supersolution of \eqref{eikonal}. If $u\leq v$ on $\partial \Omega$, then $u\leq v$ in $\Oba$. 
\end{thm}




We would also like to remark that the comparison results may be applied to the situation with only partial boundary data (or the so-called boundary condition in the viscosity sense). In general, even though $g$ is prescribed on $\partial \Omega$, a Monge solution may only fulfill the condition on a subset of $\partial \Omega$. 
We discuss such situation in detail in Section~\ref{sec:bdry-loss}.

We concluding this section by remarking that the boundedness condition on $u$ and $v$ in Theorem \ref{thm comparison monge}  is essential. In fact, without assuming 
the boundedness, we have the following simple counterexample for uniqueness of Monge solutions. 
\begin{example}\label{ex boundedness}
Let $\X$ be the unit circle in $\R^2$, centered at $O=(0, 0)$, that is, 
\[
\X=\mathbb{S}^1
=\{(\cos\theta, \sin \theta): -\pi < \theta\leq \pi\}.
\]
Here $(\X, d)$ is a complete geodesic space with $d$ its intrinsic length metric. 
Let $\Omega=\X\setminus \{(1, 0)\}$ and $g(1, 0)=0$. Take $f: \Omega\to \R$ to be 
\[
f(\cos\theta, \sin\theta)=\begin{cases}1/\pi & \text{if $0\leq \theta\leq \pi$,}\\
1/(\theta+\pi) & \text{if $-\pi<\theta<0$}.
\end{cases}
\] 
See Figure \ref{fig:circle}. 
\begin{figure}[H]
\centering
  \centering
  \includegraphics[width=0.5\linewidth]{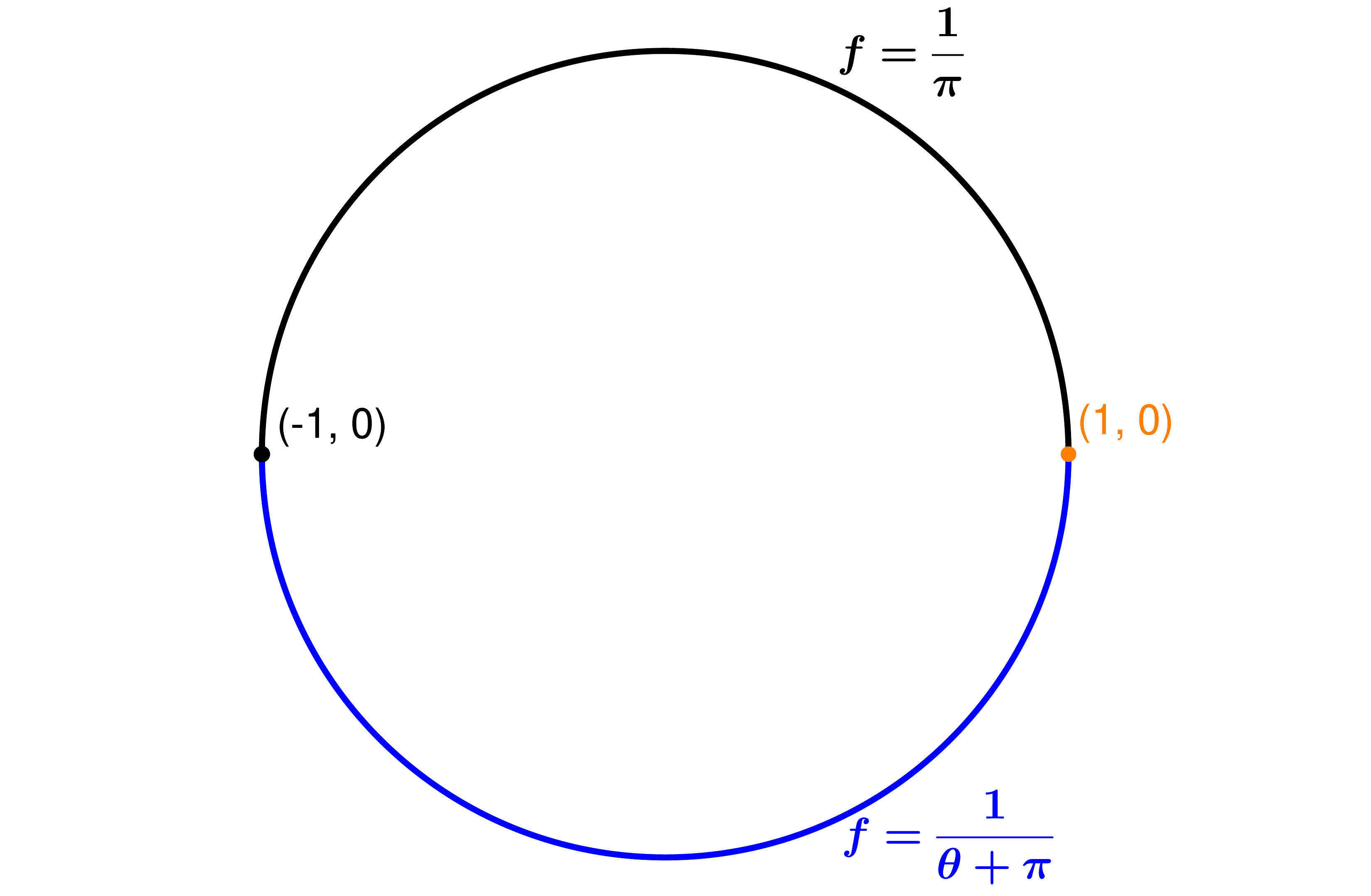}
  \caption{Example \ref{ex boundedness}}
  \label{fig:circle}
\end{figure}

Note that 
\[
u(\cos\theta, \sin\theta)=\begin{cases}\theta /\pi & \text{if $0\leq \theta\leq \pi$,}\\
\log (\pi/ \pi+\theta) & \text{if $-\pi<\theta<0$}
\end{cases}
\]
is a Monge solution of \eqref{eikonal},\,\eqref{dirichlet}, which can be derived from the formula \eqref{lax} in Theorem~\ref{exist}. 
Note that as $(-\pi,0)\ni\theta\to -\pi$, we must have $L_f((\cos\theta,\sin\theta), (-1,0))\to\infty$.
On  the other hand, we observe that 
\[
v(\cos\theta, \sin\theta)=\begin{cases}\theta /\pi & \text{if $0\leq \theta\leq \pi$,}\\
-\log (\pi/ \pi+\theta) & \text{if $-\pi<\theta<0$}
\end{cases}
\]
is also a Monge solution that also satisfies \eqref{bdry regularity}. 
This example shows that in general we may have multiple solutions, continuous but unbounded in $(\Oba, L_f)$. 
In this example, once we equip $\X=\overline{\Omega}$ with the metric $L_f$, the topology on $\X$ 
changes so that $\Omega$ now is homeomorphic to the Euclidean set $(-\infty,0)\cup(0,1]$, with boundary $\{0\}$. In effect,
$\Omega$ is no longer connected with respect to the topology generated by the metric $L_f$, allowing us to
construct two distinct solutions.
\end{example}

\section{Existence of Monge solutions}\label{sec:existence}

In this section, we establish an optimal control interpretation of the Dirichlet boundary condition \eqref{dirichlet}
for the problem~\eqref{eikonal} and prove our main result, Theorem \ref{exist}. 
 We also discuss the case when the boundary compatibility condition \eqref{growth} fails to hold and turn it into another Dirichlet problem of the same type but with reduced boundary data.

\subsection{Optimal control formulation for the Dirichlet problem}\label{sec:control}
Recall that the standing assumption \eqref{c-integrable} still holds. 
In other words, there exist curves in $\Oba$ connecting any $x, y\in \Oba$ 
such that $L_f(x,y)<\infty$. It is a condition on $\Omega$ as well as the regularity of $f$. 


\begin{lem}[Lipschitz continuity in $L_f$]\label{lem conti}
Let $(\X, d)$ be a complete length space, $\Omega\subsetneq\X$ be a bounded domain, $f$ satisfy \eqref{f-lower},
and $g: \partial \Omega\to \R$ be bounded. Then the function $u$, as defined by \eqref{lax},
satisfies
\begin{equation}\label{dpp-sub}
u(x)\le u(y)+L_f(x, y)\quad \text{for all $x, y\in \Oba$}. 
\end{equation}
In particular, $u$ is Lipschitz in $\Oba$ with respect to $L_f$ and hence locally bounded with
respect to the topology generated by $L_f$.
Moreover, if \eqref{ol-uc} holds, then $u$ is uniformly continuous in $\Oba$ with respect to $d$. 
\end{lem}


\begin{proof}
By Lemma \ref{lem metric-like}(3), we have for any $x, y\in \Oba$ and $z\in \partial \Omega$,
\[
u(x)\le L_f(x, z)+g(z)\leq L_f(y, z)+g(z)+L_f(x, y).
\]
Then taking infima over $z\in \partial \Omega$, we get \eqref{dpp-sub}.
As an immediate consequence, the Lipschitz continuity of $u$ in $\Oba$ with respect to $L_f$ holds. 
Also, we can obtain the uniform continuity of $u$ with respect to $d$ if we additionally assume that \eqref{ol-uc} holds.
\end{proof}

Let us now prove our main result, Theorem \ref{exist}. 


\begin{proof}[Proof of Theorem \ref{exist}]
We first show (i). By Lemma \ref{lem conti}, we see that $u$ is Lipschitz with respect to $L_f$ and is a 
Monge subsolution; indeed, 
by \eqref{dpp-sub} we have that for any 
$x_0\in \Omega$, 
\begin{equation}\label{eq:sub-opt}
u(x_0)-u(x)\leq L_f(x_0, x)\quad \text{for all $x\in \Omega$}
\end{equation}
 and thus 
\begin{equation}\label{eq:sub-opt2}
\limsup_{L_f(x, x_0)\to 0} \frac{u(x_0)-u(x)}{L_f(x_0, x)}\leq 1. 
\end{equation}
On the other hand, in view of the definition of $u$ from \eqref{lax}, for any $\vep>0$ and any $r\in (0, d(x_0,\partial\Omega))$, there exist 
$y\in \partial \Omega$ and $\gamma\in \Gamma_f(x_0, y)$ such that 
\begin{equation}\label{eq:super-opt}
u(x_0)\geq \int_\gamma f \, ds+g(y) -\vep r.
\end{equation}
Let $x$ be a point on $\gamma$ with $d(x_0, x)=r$ and denote by $\gamma_x$ the portion of $\gamma$ between $x_0$ and $x$. Then by \eqref{lax} again, we deduce that
\[
u(x_0)\geq \int_{\gamma_x} f\, ds+u(x)-\vep d(x_0, x)\geq u(x)+L_f(x_0, x)-\vep d(x_0, x). 
\]
It follows that 
\[
\frac{u(x_0)-u(x)}{L_f(x_0, x)}\geq 1-\vep \frac{d(x_0, x)}{L_f(x_0, x)},
\]
which, by \eqref{f-lower}, yields 
\[
\limsup_{L_f(x, x_0)\to 0}\frac{u(x_0)-u(x)}{L_f(x_0, x)}\geq 1-\frac{\vep}{\alpha}. 
\]
Due to the arbitrariness of $\vep>0$ and $x_0\in \Omega$, we obtain the Monge supersolution property of $u$. We have thus shown that $u$ is a Monge solution of \eqref{eikonal} in the sense of Definition~\ref{Monge def}. 

Let us show (ii) and (iii). We prove \eqref{bdry regularity} under the additional compatibility condition \eqref{growth} on the boundary value $g$. 
Indeed, fixing $x\in \Oba$ and $y\in \partial \Omega$ arbitrarily, 
by the definition of $u$ in \eqref{lax} we can see that 
\begin{equation}\label{bdry reg1}
u(x)-g(y)\leq L_f(x, y)
\end{equation}
For any $\delta>0$ small, the definition of $u$ also implies the existence of 
 $y_\delta\in \partial \Omega$ such that 
\begin{equation} \label{bdry continuity3}
u(x)\geq L_f(x, y_\delta)+g(y_{\delta})-\max\{\delta, d(x, \partial\Omega)\}. 
\end{equation}
Since the compatibility condition \eqref{growth} yields
\begin{equation*} 
g(y)\leq g(y_\delta)+L_f(y, y_\delta), 
\end{equation*}
applying Lemma \ref{lem metric-like} we deduce from \eqref{bdry continuity3} that
\begin{equation*} 
u(x)-g(y)\geq L_f(x, y_\delta)-L_f(y, y_\delta)-d(x, \partial\Omega)\geq -L_f(x, y)-\max\{\delta, d(x, \partial\Omega)\}. 
\end{equation*}
Combining \eqref{bdry reg1}  with the above inequality we obtain 
\begin{equation}\label{bdry continuity}
|u(x)-g(y)|\leq L_f(x, y)+\max\{\delta, d(x, \partial\Omega)\}
\end{equation}
for all $x\in \Omega$ and $y\in \partial \Omega$. Therefore, 
\begin{equation}\label{bdry continuity2}
\begin{aligned}
\sup\{|u(x)-g(y)|:\  & x\in \Oba,\  y\in \partial \Omega,\ d(x, y)\leq \delta\}\\
&\leq \sup\{L_f(x, y): x, y\in \Oba\setminus \Omega_\delta,\ d(x, y)\leq \delta\}+\delta.
\end{aligned}
\end{equation}
By the additional assumption \eqref{ol-uc2}, we obtain \eqref{bdry regularity} by passing to the limit of \eqref{bdry continuity2} as 
$\delta\to 0$. In particular, we have $u=g$ on $\partial \Omega$. The uniqueness of solutions follow from the comparison principle, Theorem~\ref{thm comparison monge}. Suppose that there is another such solution $v$ satisfying~\eqref{bdry regularity}. 
For any $x\in \Oba$ with $d(x, \partial \Omega)\leq \delta$, we can find $y\in \partial\Omega$ such that $d(x, y)\leq 2\delta$ and 
\[
\begin{aligned}
&\sup\left\{u(x)-v(x): x\in \Oba, \ d(x, \pO)\leq \delta \right\}\\
&\leq \sup\{|u(x)-g(y)|+|v(x)-g(y)|: x\in \Oba, y\in \partial \Omega, d(x, y)\leq 2\delta\}.\end{aligned}
\]
By \eqref{bdry regularity}, this yields \eqref{bdry verify monge}. Using Theorem~\ref{thm comparison monge}, we end up with $u=v$ in $\Oba$, which completes the proof of (ii).

The statement (iii) can be immediately proved, since \eqref{ol-uc} implies \eqref{ol-uc2}. We obtain the uniform continuity of $u$ with respect to $d$ by \eqref{ol-uc} and Lemma ~\ref{lem conti}. 
\end{proof}

\begin{rmk}\label{rem exist}
Following Theorem \ref{thm comparison geodesic}, we see that if $(\X, L_f)$ is complete and locally compact, 
then the function $u$ obtained by \eqref{lax} is the 
unique Monge solution of \eqref{eikonal}\eqref{dirichlet} provided that \eqref{growth} holds. We do not need to adopt interpretation \eqref{bdry regularity} for \eqref{dirichlet}.
\end{rmk}

Following Example \ref{ex no-bilip}, we now give an explicit example of unique Monge solution of the Dirichlet problem in 
one dimension with $L^p$ inhomogeneous term. 

\begin{example}\label{ex dis eikonal}
Let $(\X, d)$, $\Omega$ and $f$ be given as in Example \ref{ex no-bilip}. We recall that $f\in L^p(\Omega)$ with $p\in [1, 2)$. 
Setting $g(\pm 1)=0$, by Theorem \ref{exist}(iii) we can show that the unique Monge solution of the Dirichlet problem is as in \eqref{eq ex dis eikonal}.
\end{example}


The statement (i) in Theorem \ref{exist} is a very general existence result for Monge solutions. It allows us to have a discontinuous (with respect to $d$)
Monge solution of \eqref{eikonal} even
when \eqref{ol-uc} fails to hold. In this case, under the assumptions of (ii), we have unique existence of a solution that is Lipschitz with respect to $L_f$ but possibly discontinuous in the metric $d$. 
Below we give a typical example illustrating such behavior based on Example~\ref{ex singular}.

\begin{example}\label{ex disc sol}
Let $(\X, d)$ and $f$ be given as in Example \ref{ex singular}. Let $\Omega=\X\setminus \{Q_1\}$; in other words, we set $\partial \Omega=\{Q_1\}$. Also, we take $g=0$ at $Q_1$. We have shown in Example \ref{ex singular} that $\X=\Oba$ is complete with respect to the metric $L_f$. 
Applying Theorem \ref{exist}(ii), 
we obtain a unique Monge solution $u$ of \eqref{eikonal}\eqref{dirichlet} satisfying \eqref{bdry regularity}. In particular, by \eqref{lax} we have $u(O)=2$, $u(Q_j)=3-1/j$ for all $j\geq 2$. This clearly shows that 
$u$ is not continuous at $O$ with respect to the metric $d$. However, 
$u$ is Lipschitz continuous with respect to the metric $L_f$. 
\end{example}

Let us present an additional property of the Monge solution $u$, which is a direct generalization of \cite[Proposition 4.15]{LShZ}.

\begin{prop}[Additional regularity]\label{prop monge-reg}
Let $(\X, d)$ be a complete length space and $\Omega\subsetneq\X$ be a bounded domain.  Assume that 
$f$ satisfies \eqref{f-lower},~\eqref{c-integrable} and  
that $g$ satisfies \eqref{growth} and \eqref{ol-uc2} holds. 
Then any $u$ that satisfies \eqref{bdry regularity} is the unique Monge solution of \eqref{eikonal} if and only if  
\begin{equation}\label{eq semicon sol}
|\nabla_f u|=1\quad \text{and} \quad |\nabla_f^+ u|\leq |\nabla_f^- u|\quad \text{ in $\Omega$.}
\end{equation}
Here, for any locally bounded function 
$u: \Omega\to \R$ and $x\in \Omega$, we set 
\begin{equation}\label{new slope}
\begin{aligned}
&|\nabla_f u|(x)=\limsup_{L_f(x, y)\to 0} \frac{|u(y)-u(x)|}{L_f(x,y)},\\
&|\nabla_f^+ u|(x)=\limsup_{L_f(x, y)\to 0} \frac{\max\{u(y)-u(x), 0\}}{L_f(x,y)}.
\end{aligned}
\end{equation}
\end{prop}

\begin{proof}
The implication $\Leftarrow$ is obvious, as \eqref{eq semicon sol} also implies $|\nabla_f^- u|=1$ in $\Omega$. It suffices to prove $\Rightarrow$. The unique Monge solution $u$ can be expressed by \eqref{lax}. By Lemma \ref{lem conti},  it satisfies \eqref{eq:sub-opt}. It follows that for every fixed $x_0\in \Omega$, we have $u(x)\leq u(x_0)+L_f(x_0, x)$ for for $x\in \Omega$,
which immediately yields $|\nabla_f^+ u|(x_0)\leq 1$. Since 
\[
|\nabla_f u|=\max\left\{|\nabla_f^+ u|,\ |\nabla_f^- u|\right\},
\]
we immediately get \eqref{eq semicon sol}. 
\end{proof}

\subsection{Loss of boundary data and reduced Dirichlet problem}\label{sec:bdry-loss}

In general, the condition \eqref{growth} may not hold. As a result, it is possible that $u=g$ holds only on a subset of $\partial\Omega$. This set may also depend on $f$ but we write $\Sigma_g$ to emphasize that it is the part where the Dirichlet condition is maintained. 
The loss of boundary data occurs on $\partial \Omega \setminus \Sigma_g$. 
 At points $x_0\in\partial\Omega\setminus\Sigma_g$ the Monge condition~\eqref{monge-sub} holds, and so
the function $u$ constructed in~\eqref{lax} solves the problem in the new domain $\Oba\setminus\Sigma_g$. 
We thus can use the comparison results above with $\partial \Omega$ replaced by $\Sigma_g$ to guarantee the uniqueness of solutions. The following typical, well-understood example in $\R$ reveals such a situation.

\begin{example}\label{ex bdry-loss1}
Let $\Omega=(0, 1)\subset \R$ and $f\equiv 1$ in $\Omega$. Set $g(0)=0$ and $g(1)=2$. It turns out that there are no solutions that are continuous in $[0, 1]$ and satisfy both~\eqref{eikonal} and~\eqref{dirichlet} simultaneously. However, $u(x)=x$  is the unique solution that satisfies only the partial boundary condition $g(0)=0$. In fact, $|\nabla^- u|(1)=1$ holds and we can still use the comparison principle if we consider $\Omega=(0, 1]$ and $\partial\Omega= \{0\}$. One can certainly take another solution $u(x)=x+1$ for $[0, 1]$, which is a Monge solution in $[0, 1)$ and satisfies the  boundary condition $u(1)=2$. However, this solution 
does not satisfy the Monge condition~\eqref{monge-sub} at $0$,
and moreover, 
such solutions are not necessarily unique from the PDE viewpoint; $u(x)=-x+3$ is also a Monge solution in $[0, 1)$ with $u(1)=2$. 
Indeed, it is the \emph{only} Monge solution in $[0,1)$ with boundary data $u(1)=2$.
We resolve the uniqueness issue by taking $\Sigma_g=\{0\}$ so that $u(x)=x$ is the only solution in $\Omega=(0, 1]$ satisfying $u(0)=0$. The function $u(x)=x$ is regarded as a natural choice of the unique solution here also due to the optimal control interpretation given by \eqref{lax}.
\end{example}

Let us provide a more general result for the case when \eqref{growth} fails to hold and $u\neq g$ on $\partial \Omega$. 
As explained at the end of Section \ref{sec:comparison}, we still expect that there exists a solution $u$ such that $u=g$ on a subset $\Sigma_g$ of $\partial \Omega$. In fact, we take
\begin{equation}\label{effective bdry}
\Sigma_g=\left\{x\in \partial \Omega: g(x)\leq \inf_{y\in \partial \Omega}\{g(y)+L_f(x, y)\}\right\}.
\end{equation}
By \eqref{lax}, it is clear that
\begin{equation}\label{effective bdry2}
\Sigma_g=\left\{x\in \partial \Omega: g(x)\leq u(x)\right\}=\left\{x\in \partial \Omega: g(x)= u(x)\right\}.
\end{equation}
In general, the set $\Sigma_g$ may be empty. However, $\Sigma_g\neq \emptyset$ if 
$\partial \Omega$ is compact with respect to $L_f$ and 
and $g\in LSC_{L}(\partial \Omega)$, or if $\partial\Omega$ is compact with respect to $d$ and $g\in LSC_d(\partial\Omega)$. 
In either case, $\Sigma_g$ contains the minimizers of $g$ on $\partial \Omega$. Also, $\Sigma_g$ is a closed set with respect to the metric $L_f$.  
Using $\Sigma_g$ we can reduce the original Dirichlet problem to
\begin{equation}\label{eikonal-g}
|\nabla u|(x)=f(x) \quad \text{in $\Oba\setminus \Sigma_g$}
\end{equation}
with boundary condition 
\begin{equation}\label{dirichlet-g}
u=g \quad \text{for $x\in\Sigma_g$.}
\end{equation}

It turns out that $u$ given by \eqref{lax} is the unique Monge solution of \eqref{eikonal-g},\, \eqref{dirichlet-g} when $\Sigma_g\neq \emptyset$. Here, 
for these new Monge solutions, we certainly need to 
extend the definition of Monge solutions to $\Oba\setminus \Sigma_g$. More precisely, we say a locally bounded function $u: \Oba\setminus \Sigma_g\to \R$ is a Monge solution of \eqref{eikonal-g} if it satisfies the property in Definition \ref{Monge def} with 
$\Oba\setminus \Sigma_g$ playing the role of the domain $\Omega$ there, and with \eqref{monge-sub} replaced by 
\begin{equation}\label{monge-g}
\limsup_{x\in \Oba\setminus \Sigma_g,\ L_f(x, x_0)\to 0}\frac{u(x_0)-u(x)}{L_f(x_0, x)}= 1.
\end{equation}
One can define Monge subsolutions and supersolutions of \eqref{eikonal-g} in a similar way. 

\begin{prop}[Well-posedness with possible loss of boundary data]\label{prop bdry loss}
Let $(\X, d)$ be a complete length space and $\Omega\subsetneq\X$ be a bounded domain. 
Assume that \eqref{f-lower} and \eqref{c-integrable} hold. Let $g: \partial \Omega\to \R$ be bounded 
and $u$ be defined by \eqref{lax}. Let $\Sigma_g$ be given by \eqref{effective bdry}. Then $u$ is a Monge solution of 
\eqref{eikonal-g}, \,\eqref{dirichlet-g}. If in addition $\Sigma_g\neq \emptyset$ and  \eqref{ol-uc2} holds, then $u$ is the unique Lipschitz (with respect to $L_f$) Monge solution of the Dirichlet problem \eqref{eikonal-g}\eqref{dirichlet-g} satisfying 
\begin{equation}\label{bdry regularity-g}
\sup\{|u(x)-g(y)|: x\in \Oba,\ y\in \Sigma_g,\ d(x, y)\leq \delta\}\to 0\quad \text{as $\delta\to 0$}.
\end{equation}
\end{prop}
\begin{proof}
We have shown in Theorem \ref{exist} that $u$ is a Monge solution in $\Omega$. Part of the proof holds also for boundary points. In fact, by \eqref{dpp-sub} we get, for any fixed $x_0\in \partial \Omega$, 
\[
u(x_0)-u(x)\leq L_f(x_0, x) \quad \text{for all $x\in \Oba$,}
\]
which yields \eqref{eq:sub-opt2} with $x_0\in \partial \Omega$. Now, for any $x_0\in \partial  \Omega\setminus \Sigma_g$, since $u(x_0)<g(x_0)$, by \eqref{lax}, for any $\vep>0$ we still have \eqref{eq:super-opt} for some $y\in \partial \Omega$ and $\gamma\in \Gamma_f(x_0, y)$. Then we can follow the same argument in the proof of Theorem \ref{exist} to prove that 
\[
\limsup_{ x\in \Oba\setminus \Sigma_g,\ L_f(x, x_0)\to 0}\frac{u(x_0)-u(x)}{L_f(x_0, x)}\geq 1.
\]
Hence, \eqref{monge-g} holds for all $x_0\in \Oba\setminus \Sigma_g$. In view of \eqref{effective bdry2}, we also have $u=g$ on $\Sigma_g$. Thus $u$ is a Monge solution of \eqref{eikonal-g}, \,\eqref{dirichlet-g}.

If $\Sigma_g\neq \emptyset$, we can further obtain \eqref{bdry regularity-g}. Indeed, we have \eqref{bdry reg1}--\eqref{bdry continuity2} with $\Omega$ and $\partial\Omega$ replaced by $\Oba\setminus \Sigma_g$ and $\Sigma_g$ respectively. Then we deduce \eqref{bdry regularity-g} under the condition \eqref{ol-uc2}.

The comparison principle, Theorem \ref{thm comparison monge}, can also be extended to the current case with $\Omega$ and $\partial\Omega$ replaced by $\Oba\setminus \Sigma_g$ and $\Sigma_g$ respectively. Hence, the uniqueness of Monge solutions of \eqref{eikonal-g}\eqref{dirichlet-g} with \eqref{bdry regularity-g} holds. 
\end{proof}

\begin{rmk}
As $u<g$ on $\partial\Omega \setminus \Sigma_g$, by \eqref{dpp-sub} we have $u(x)<g(y)+L_f(x, y)$
for all $x\in \Oba$ and $y\in \partial\Omega \setminus \Sigma_g$. This means that the optimal control formula \eqref{lax} can be rewritten as 
\[
u(x)=\inf_{y\in \Sigma_g}\{L_f(x, y)+g(y)\}
\]
if $\Sigma_g\neq \emptyset$. 
In other words, increasing the value of $g$ at points in $\partial\Omega \setminus \Sigma_g$ 
will not change the solution $u$.
\end{rmk}

In the Euclidean space, when $\Sigma_g\neq \partial\Omega$ and $f\in C(\Omega)$, we usually relax the meaning of the Dirichlet condition on $\partial \Omega\setminus \Sigma_g$, requiring $u$ to satisfy
\begin{equation}\label{bdry vis}
|\nabla u|\geq f \quad \text{on $\partial \Omega\setminus \Sigma_g$}
\end{equation}
in the viscosity sense. Under appropriate assumptions on the regularity of $\Omega$, one can show that the generalized Dirichlet (or state constraint) problem still admits a unique viscosity solution \cite{So1, So2, I7}.
Our result in Proposition \ref{prop bdry loss} essentially handles this type of generalized Dirichlet boundary problems in metric spaces. 
In fact, we have \eqref{monge-g} at each $x_0\in \partial \Omega\setminus \Sigma_g$, which corresponds to the viscosity inequality \eqref{bdry vis}.  
Such topological change was also observed in \cite[Remark 5.10]{LNa} for evolutionary Hamilton-Jacobi equations in metric spaces.




\section{Existence of admissible curves and regularity of solutions}\label{sec:admissible}

We have seen that the existence of admissible curves \eqref{c-integrable} plays an important role in finding the Monge solutions of \eqref{eikonal}. The consistency of local topology condition \eqref{ol-uc}, which guarantees the continuity with 
respect to the original metric $d$ and uniqueness of the Monge solution $u$ of the Dirichlet problem, is also related to the existence of admissible curves. We already provided an example, Example \ref{ex singular}, to show that \eqref{ol-uc} fails to hold in general. It is natural to ask under what assumptions we can obtain \eqref{c-integrable} and \eqref{ol-uc}. In this section, we attempt to answer this question. 

We introduce some notations and definitions for our use in this section. Let $(\X, d, \mu)$ be a metric measure space with $\mu$ being a locally finite Borel measure. 
We use $B$ to denote an open ball $B_r(x)$ when the center $x$ and radius $r>0$ are irrelevant to our analysis. In this case, we also write $\lambda B=B_{\lambda r} (x)$ for any $\lambda>0$ for simplicity of notation.


Let $p \ge 1$. Let $\Gamma$ be a collection of nonconstant rectifiable curves in $(\X,d,\mu)$. We say that a curve connecting $x,y\in X$ is a $C$-quasiconvex curve if the length of the curve is at most $C\, d(x,y)$.
Let $\Gamma(x, y; C)$ denote the collection of $C$-quasiconvex curves connecting $x,y\in \X$.
A metric space is said to be quasiconvex if there exists $C\ge 1$ such that every pair of points 
$x,y\in \X$ can be connected by a $C$-quasiconvex curve.

Let $F(\Gamma)$ be the family of all Borel measurable functions $\rho:X\to [0,\infty]$ such that $\int_\gamma \rho\ge 1$ for every $\gamma \in \Gamma$. For each $1\le p<\infty$, the $p$-modulus of $\Gamma$ is defined as 
\[
\mo_p(\Gamma)=\inf_{\rho\in F(\Gamma)} \int_\X \rho^p\ d\mu,
\]
and the $\infty$-modulus of $\Gamma$ is defined as
\[
\mo_\infty(\Gamma)=\inf_{\rho\in F(\Gamma)} \|\rho\|_{L^\infty(\X)}.
\]

For any given function $u: \X\to \mathbb{R}$, a Borel function $\rho: \X\to [0,\infty]$ is said to be a {$p$-weak upper gradient}  of $u$ if
\[|u(\gamma(a))-u(\gamma(b))|\le \int_{\gamma}\rho\ ds\]
holds for all rectifiable paths  $\gamma:[a,b]\to \X$ outside a family  of curves with $p$-modulus zero. Consult \cite{HKSTBook} for an introduction about modulus of curve family and upper gradients.

Let $1\le p<\infty$. A metric measure space $(\X, d, \mu)$ is said to support a {$p$-Poincar\'e inequality} if there exist constants $C>0$ and $ \lambda\ge1$ such that for every measurable function
$u: \X\to \mathbb{R}$ and every upper gradient $\rho: \X \to [0,\infty]$ of $u$, the pair $(u, \rho)$ satisfies
\begin{equation}\label{poincareinequality}
\fint_B |u-u_B|\, d\mu\le C\diam B\left(\fint_{\lambda B}\rho^p\, d\mu\right)^{1/p}
\end{equation}
on every ball $B\subset X$. If $p=\infty$, the right hand side above is replaced by $C\diam B\|g\|_{L^\infty(\lambda B)}.$ Here and in the sequel, for an integrable function $f$ and a measurable set $A$ of finite measure, we take
\[
\fint_A f\, d\mu:={1\over \mu(A)}\int_A f\, d\mu.
\]   
We say that a metric measure space $(\X, d,\mu)$ is a PI-space if the measure $\mu$ is doubling and the space 
supports a $p$-Poincar\'e inequality for some $p\ge 1$.



It turns out that~\eqref{c-integrable} and \eqref{ol-uc} hold on a broad class of 
PI-spaces.

\begin{thm}[Regularity in $p$-PI spaces]\label{thm doubling}
Let $(\X, d, \mu)$ be a complete doubling metric measure space supporting a $p$-Poincar\'e inequality with $\max\{1,Q\}<p<\infty$, where $Q>0$ is the homogeneous dimension satisfying \eqref{hdim}. 
Then there exists a constant $C>0$ depending on the constants of PI-space and $R>0$ such that for every function $u$ 
and its $p$-weak upper gradient $\rho\in L^p(\X)$, the following 
inequality holds for each open ball $B_R(z)$ and every $x,y\in B_R(z)$:
\begin{equation}\label{Holder}
|u(x)-u(y)|\le Cd(x,y)^{1-\frac{Q}{p}}\|\rho\|_{L^p(B_{5\lambda R}(z))},
\end{equation}
where $\lambda$ is the scaling constant in Poincar\'e inequality. 
Furthermore, for any pair of distinct points $x, y\in B_R\subset \X$, there exists $C>0$ depending on the constants of PI-space and $R>0$ such that 
\begin{equation}\label{modulus est}
\mo_p(\Gamma(x, y))\geq {C\over  d(x, y)^{p-Q}},
\end{equation}
where $\Gamma(x, y)$ is the family of all rectifiable curves in $B_{5\lambda R}$ joining $x$ and $y$.
\end{thm}

A stronger version of the above result on Ahlfors $Q$-regular space with $Q>1$ can be 
found in~\cite[Theorem 5.1]{DJS1}. 

\begin{proof}[Proof of Theorem \ref{thm doubling}]
The proof of \eqref{Holder} can be found in \cite[Theorem 9.2.14]{HKSTBook}.  We only show the modulus estimate \eqref{modulus est} below. We only consider a pair of distinct points $x,y\in B_R\subset \X$ such that $\mo_p(\Gamma(x, y))<\infty$. Otherwise, the proof is complete. We pick $\rho\in F(\Gamma(x,y))$, that is, $\int_\gamma \rho\, ds \ge 1$ for every rectifiable curve $\gamma\in \Gamma(x,y)$. Define 
\[
v_\rho(z)=\inf_{\gamma\in \Gamma(x,z)} \int_{\gamma} \rho\ ds,
\]
and then $v_\rho$ is measurable \cite[Corollary 1.10]{JJRRS} and $\rho$ is an upper gradient of $v_\rho$. In fact, it is clear that $v_\rho(x)=0$ and 
\[
v_\rho(z)\le \int_\gamma \rho\ ds\quad \text{for any $\gamma\in \Gamma(x,z)$}.
\]
Therefore $\rho$ is an upper gradient for $v_\rho$ in $\X$. 
It follows that $\rho$ is also an upper gradient for $\min\{v_\rho,2\}$. 

Let $\eta: \X\to \mathbb{R}$ be a Lipschitz function satisfying that $0\le \eta\le 1$, $\eta=1$ on $5\lambda B$ and $\eta=0$ on $\X\setminus 10\lambda B$. Define a function $u=\eta\min\{v_\rho,2\}$. It easily follows that $u(x)=0$ and $u(y)\ge 1$. One can also verify that ${g}=2|\nabla \eta|+\rho$ is an upper gradient of $u$. In particular, since $\eta=1$ on $5\lambda B$, we obtain $\rho$ is an upper gradient for $u$ on $5\lambda B$. Hence, it follows that
\[
\begin{aligned}
1\le |u(x)-u(y)|&\le Cd(x,y)^{1-\frac{Q}{p}}\|\rho\|_{L^p(B_{5\lambda R}(z))}.
\end{aligned}
\]
Since $\rho\in F(\Gamma(x,y))$ is arbitrary, by definition we obtain \eqref{modulus est}. 
\end{proof}

This theorem implies the following result.

\begin{prop}[H\"older regularity]\label{prop regular}
Let $(\Oba, d, \mu)$ be a  complete bounded metric measure space with $\mu$ a doubling measure.
Assume that $\Oba$ and $f$ satisfy (A1). 
Then for any distinct $x, y\in \Oba$, there exists an admissible 
curve $\gamma\in \Gamma_f(x, y)$, 
and in particular, \eqref{c-integrable} holds. Furthermore, there exists a constant $C_f>0$ such that 
\begin{equation}\label{Ld-holder}
L_f(x,y)\le C_f\, d(x,y)^{1-\frac{Q}{p}}\quad \text{for all $x, y\in\Oba$,}
\end{equation}
which implies that~\eqref{ol-uc} holds.
\end{prop} 

\begin{proof}

Note that for any $f\in L^p(\Oba)$ with $1\le p< \infty$,  the collection $\Gamma$ of nonconstant rectifiable curves such that $\int_\gamma f\ ds=\infty$ has $p$-modulus zero; see for instance \cite[Lemma 5.2.8]{HKSTBook}.  By \eqref{modulus est} in Theorem \ref{thm doubling}, we get $\mo_p(\Gamma(x, y))>0$ for any distinct points $x, y\in \Oba$, 
which implies the existence of rectifiable curves joining $x$ and $y$ on which $f$ is integrable. 

For any fixed $x\in \Oba$, writing $v(y)=\inf_{\gamma\in \Gamma_f(x,y)} \int_\gamma f\ ds$ for $y\in \overline{\Omega}$, 
we can deduce that $v$ is measurable and locally $p$-integrable by \cite[Theorem 1.11]{JJRRS} and 
$f\in L^p(\Oba)$ is an upper gradient of $v$. 
Applying Theorem~\ref{thm doubling}, we are led to 
\[
L_f(x,y)=v(y)=v(y)-v(x)\le Cd(x,y)^{1-\frac{Q}{p}}\|f\|_{L^p(\Oba)}.
\] 
Hence, we also obtain \eqref{ol-uc}.
\end{proof}


In general, if $f\notin L^p(\Omega)$ for $p>Q$, then \eqref{c-integrable} may not hold in general, as can be seen in the following simple example, which is an adaptation of Example \ref{ex no-bilip}. 

\begin{example}\label{ex non-integrable}
Let $\X=\R$ with $d$ being the Euclidean metric and $\mu$ being the one-dimensional Lebesgue measure. Let $\Omega=(-1, 1)$ and 
\[
f(x)={1\over x}, \quad x\in (-1, 1)\setminus\{0\}.
\]
It is clear that $f\notin L^p(\Omega)$ for any $p\ge 1$. Then for any $x\in (-1, 1)\setminus \{0\}$, $\Gamma_f(x, 0)=\emptyset$ and $L_f(x, 0)=\infty$.
\end{example}

The critical case $p=Q$ is complicated. 
For $\R^n$ with $n=1$, we can prove the conditions~\eqref{c-integrable} and \eqref{ol-uc} if $p\ge 1$, since the $p$-modulus of the curve collection containing just one non-constant curve is positive. In general, we cannot expect the results in Proposition \ref{prop regular} to always hold if $f\in L^Q(\Omega)$ but $f\notin L^p(\Omega)$ for any $p>Q$.  
On the other hand, \eqref{c-integrable} is only about the existence of one curve on which $f$ is integrable and can be obtained even for some functions $f\notin L^Q(\Omega)$.  
A simple example is as below.
\begin{example}\label{ex c-integrable}
Let $\X=\R^2$ and $\Omega=B_1(O)\subset \X$ with $O$ denoting the origin $(0, 0)$. 
That is, $\Omega$ is the unit disk in $\R^2$ centered at $O$.
We set
 \[
 e_{1}:=\{(x_1, x_2): x_1\in (0, 1), x_2=0\}.
 \]
Let $f: \Omega\to \R$ be given by 
\[
f(x)=\begin{dcases}
{1\over \sqrt{|x_1|}} &\text{if $x\in e_1$,}\\
{1\over |x|} & \text{if $x\notin e_1$ and $x\neq O$.}
\end{dcases}
\]
Observe that $f\in L^p(\Omega)$ for $1\leq p<2$ but $f\notin L^2(\Omega)$; the value of $f$ on $e_1$ 
does not affect the integral of $f$ in $\Omega$. 
 However, for each pair of distinct points $x, y\in \Oba$, one can find a curve $\gamma\in \Gamma_f(x, y)$ in $\Oba$. 
For $z\neq O$, we can take a curve connecting $z$ first to $(|z|, 0)$ along the circular arc and then to $O$ along the horizontal line segment. This choice enables us to guarantee $L_f(O, z)\le \pi+2$, which shows that $\Oba$ a bounded metric space with respect to the metric $L_f$. 
 The topology induced by $L_f$, however, is
 distinct from the Euclidean topology. Indeed, for $z_j:=(0,1/j)$ and any curve $\gamma$ joining $O$ and $z_j$, denoting by $\ul{\gamma}$ the portion of $\gamma$ in the Euclidean disk $B_{1/j}(z_j)$, we have $\int_\gamma f\, ds\ge \int_{\ul{\gamma}} f\, ds \geq \ell(\ul{\gamma})/(2j)\ge 1/2$, which yields $L_f(O, z_j)\ge 1/2$ for all $j\geq 1$.
Hence, $L_f(O, z_j)\not\to 0$ as $j\to \infty$, although $z_j$ converges to $O$ in the Euclidean topology.
(As the metric $L_f$ is locally bi-Lipschitz equivalent to the Euclidean metric in $\Oba\setminus\{O\}$ as well, we see
that this sequence cannot converge in $\Oba$ with respect to $L_f$.) 

\end{example}

The following result, which can be found in \cite[Theorem 3.1]{DJS1} and \cite[Theorem 2.10]{DJS2}, is a variant of Theorem \ref{thm doubling} for the case $p=\infty$. 

\begin{thm}[Regularity in $\infty$-PI spaces]\label{thm infty-regular}
Let $(\X, d, \mu)$ be a complete doubling metric measure space. 
Then the following statements are equivalent. 
\begin{enumerate}
\item $\X$ supports an $\infty$-Poincar\'e inequality.
\item There exists $C\geq 1$ such that 
$\mo_\infty(\Gamma(x, y; C))>0$ for all $x, y\in \X$ satisfying $d(x, y)>0$, where 
we recall that $\Gamma(x, y; C)$ denotes the collection of $C$-quasiconvex curves connecting $x,y$. 
\item There exists $C\geq 1$ such that whenever $N\subset \X$ with $\mu(N)=0$ and $x, y\in \X$ with $x\neq y$, there is a quasiconvex curve $\gamma$ joining $x, y$ such that $\ell(\gamma)\leq Cd(x, y)$ and $\cH^1(\gamma^{-1}(N))=0$, where $\cH^1$ denotes the one-dimensional Hausdorff measure. 
\end{enumerate}
\end{thm}

The following is the version of Proposition \ref{prop regular} corresponding to $p=\infty$. 

\begin{prop}[Lipschitz regularity]\label{prop regular 2}
Let $(\Oba, d, \mu)$ be a complete bounded metric measure space with $\mu$ a doubling measure. 
Assume that $\Oba$ and $f$ satisfy (A2). 
Then for any distinct $x, y\in \Oba$, there exists  
a curve $\gamma\in \Gamma(x, y; C)\cap \Gamma_f(x, y)$ in $\Oba$. 
 In particular, \eqref{c-integrable} and \eqref{ol-uc} hold with
\begin{equation}\label{eq prop regular 2}
L_f(x, y)\leq C\Vert f\Vert_{L^\infty(\Omega)} d(x, y)\quad\text{for all $x, y\in \Oba$.} 
\end{equation}
\end{prop} 

\begin{proof}
It is not difficult to see that Theorem \ref{thm infty-regular} (2) implies \eqref{c-integrable}. The proof of \eqref{ol-uc} involves the condition (3) in Theorem \ref{thm infty-regular}. In fact, since $f\in L^\infty(\Oba)$, there exist a null set $N_f$ 
such that 
\begin{equation}\label{ess bound}
f\leq \Vert f\Vert_{L^\infty(\Oba)} \quad \text{in $\Oba\setminus N_f$.} 
\end{equation}
Then by Theorem \ref{thm infty-regular}, for any $x, y\in \Oba$, we can take a curve $\gamma$ in $\Oba$ joining them with  $\ell(\gamma)\leq Cd(x, y)$ and $\cH^1(\gamma^{-1}(N_f))=0$, which yields \eqref{eq prop regular 2}.
Thus \eqref{ol-uc} clearly holds. 
\end{proof}

Since the Monge solution $u$ defined by \eqref{lax} is Lipschitz in the metric $L_f$, as shown in Theorem \ref{exist}, we obtain Theorem \ref{thm holder} as an immediate consequence of Proposition \ref{prop regular} and Proposition \ref{prop regular 2}. Under the same assumptions, it is not difficult to show that $u$ is in the Sobolev class $N^{1, p}(\Oba)$, that is, $u\in L^p(\Oba)$ and $u$ has an upper gradient in $L^p(\Oba)$. In fact, by \eqref{ol-fun} and \eqref{dpp-sub} we see that $f\in L^p(\Oba)$ is an upper gradient of $u$.


\section{Solutions for admissible curves transversal to null sets} 
\label{sec:lp}


In this section, we assume that $f$ satisfies \eqref{f-lower} and $(\Oba, d, \mu)$ is a complete doubling metric measure space with homogeneous dimention $Q>0$. Assuming either of the additional regularity assumptions (A1) and (A2)
as stated in the discussion following~\eqref{hdim}, 
we consider a slightly different optical length function $\tilde{L}_{f}$ and the associated solutions using curves transversal to null sets in $\Oba$. This change improves the stability of solutions with respect to the perturbation on $f$. We refer to \cite{NN, Mak} for stability results on time-dependent Hamilton-Jacobi equations in metric spaces.

\subsection{Optical length with transversal curves} 


Let ${\Gamma}_{f}^{N}(x, y)$ be the collection of arc-length parametrized rectifiable curves in $\X$ connecting $x, y\in \Oba$ with $\int_\gamma f\ ds<\infty$ and transversal to a given null set $N\subset \Oba$, as given in \eqref{cf}.
The associated optical length function 
 is defined by \eqref{olfN}. 

Under the standing assumptions of this section, we consider $f_N=f+\infty\chi_N$ if (A1) holds. Note that $f_N\in L^p(\Oba)$ with $\max\{1,Q\}<p<\infty$. 
Proposition \ref{prop regular} then implies that $\Gamma_{f_N}(x,y)\neq \emptyset$ if $x\neq y$, and we can verify that 
$\Gamma_{f_N}(x,y)=\Gamma_f^N(x,y)$ and 
\[
\int_\gamma f\, ds=\int_\gamma f_N\, ds\quad \text{for $\gamma\in \Gamma_f^N(x,y)$.}
\]
If (A2) holds, then Theorem \ref{thm infty-regular} implies that $\Gamma^N_f(x,y)\neq \emptyset$ and therefore $L^N_f(x,y)<\infty$ for any distinct points $x,y\in \Oba$. Hence, in either case we have 
\begin{equation}\label{ol-funp}
L_{f_N}(x, y)=\inf\left\{\int_\gamma f\, ds: \ \text{$\gamma\in \Gamma_{f_N}(x, y)$}\right\}=\inf\left\{\int_\gamma f\, ds: \ \text{$\gamma\in \Gamma_f^N(x, y)$}\right\}=L_f^N(x,y).
\end{equation}
We define the maximal optical length function $\tilde{L}_f: \X\times \X\to \mathbb{R}$ by \eqref{olf}.
It is easy to verify using Proposition~\ref{prop regular 2}
and Proposition~\ref{prop regular}
that the optical length function $\tilde{L}_f$ (as well as $L^N_f$ for any fixed null set $N$) also satisfies the metric properties, as discussed for $L_f$ in Lemma \ref{lem metric-like}. 

Moreover, we see that \eqref{ol-uc} holds for $\tilde{L}_f$ under either of the assumptions (A1), (A2). Indeed, it is an immediate consequence of Proposition~\ref{prop regular} when (A1) holds. In the case of (A2),  there exists a null set $N_f$ 
such that \eqref{ess bound} holds. Since $\Oba$ supports an $\infty$-Poincar\'e inequality, by Theorem \ref{thm infty-regular}, there exists $C>0$ such that for any null set $N$ and $x, y\in \Oba$ we can find a curve transversal to $N_f\cup N$ satisfying
\[
\int_\gamma f\, ds\leq C\Vert f\Vert_{L^\infty(\Omega)}\, d(x, y).
\]
Recall that $N_f=\{x\, :\, f(x)>\Vert f\Vert_{L^\infty(\overline{\Omega})}\}$. It follows that 
\begin{equation}\label{ol-upper}
\tilde{L}_f(x, y)\leq C\Vert f\Vert_{L^\infty(\Omega)}\, d(x, y)\quad \text{for all $x, y\in \Oba$}
\end{equation}
and thus \eqref{ol-uc} holds for $\tilde{L}_f$. Furthermore, under the assumption \eqref{f-lower}, together with
\eqref{lower lip}, we also obtain 
\begin{equation}\label{ol-lower}
\tilde{L}_f(x, y)\geq \alpha d(x, y) \quad \text{for all $x, y\in \Oba$.}
\end{equation}
Since $\Gamma^N_f(x, y)\subset \Gamma_f(x, y)$ for each null set $N$, it is easily seen that 
\begin{equation}\label{curve compare1}
L_f(x, y)\leq 
\tilde{L}_f(x, y) 
  \quad \text{for all $x, y\in \Oba$.}
\end{equation}

It turns out that the supremum in the definition of $\tilde{L}_f$ can be attained at a particular null set if either (A1) or (A2) holds. 
\begin{lem}\label{E}
Let $(\Oba, d, \mu)$ be a complete metric measure space with $\mu$ a doubling measure.  If (A1) or (A2) holds, then there exists a null set $E$ such that $\tilde{L}_f$ defined in \eqref{olf} satisfies
\begin{equation}\label{sup attain1}
\tilde{L}_f(x, y)=L^{E}_f(x, y)\quad \text{for all $x, y\in\Oba$.}
\end{equation}

\end{lem}


\begin{proof}
It suffices to show $\tilde{L}_f\leq L^E_f$, since the reverse inequality holds obviously. 
We first fix a null set $N_0$, and we now show
that for any fixed $x,y\in \Oba$ 
there exists a null set $N_{xy}$ containing $N_0$ such that 
\[
\tilde{L}_f(x,y)=\inf\{I_{f}(\gamma):\gamma\in \Gamma^{N_{xy}}(x, y)\}.
\]
To see this, for any $k\in \mathbb{N}$, we take a null set $N_k$ ($k\geq 1$) such that 
\[
\tilde{L}_f(x,y)\le L^{N_k}_f(x, y)+\frac{1}{k}\le L^{N_k\cup N_0}_f(x, y)+\frac{1}{k}.
\]
Set $N_{xy}=\bigcup_{k=0}^\infty N_k$. Then, 
\[
\tilde{L}_f(x,y)\ge L^{N_{xy}}_f(x, y)\ge L^{N_k}_f(x, y)\ge \tilde{L}_f(x,y)-\frac{1}{k}
\]
and thus $L_f(x,y)=L_f^{N_{xy}}(x, y)$ by letting $k\to \infty$.  The claim is proved.

Since $(\Oba,d,\mu)$ is complete and doubling, it follows that $\Oba$ is separable \cite[Lemma 4.1.13]{HKSTBook}, and therefore there exists a countable dense subset $D\subset \Oba$. For every pair of points $x_0,y_0\in D$, there exists a null set $N_{x_0y_0}$ containing $N$ such that 
\[
\tilde{L}_f(x_0,y_0)=L_f^{N_{x_0y_0}}(x_0, y_0) =\inf \{I_f(\gamma): \gamma\in {\Gamma}^{N_{x_0y_0}}(x_0, y_0)\}.
\]
We construct a null set $E$ in slightly different ways for the cases (A1) and (A2).

If (A1) holds, we set $E=\bigcup_{x_0,y_0\in D} N_{x_0y_0}$.  Then we have $\mu(E)=0$ and 
\begin{equation}
\tilde{L}_f(x_0,y_0)=L^E_f(x_0, y_0)=\inf \{I_f(\gamma): \gamma\in {\Gamma}^{E}(x_0,y_0)\}\quad\text{for all $x_0,y_0\in D$.}
\end{equation}
Fix $x, y\in \Oba$ and $\vep>0$ arbitrarily.  Set $\vep'=\left(\vep/ (4C)\right)^{\frac{p}{p-Q}}$ with constant $C>0$ from \eqref{Ld-holder} and let $x_0\in D\cap B_{\vep'}(x)$ and $y_0\in D\cap B_{\vep'}(y)$.
Proposition \ref{prop regular} implies that 
\begin{equation}\label{eq E est2}
\tilde{L}_f(x,x_0)<\vep/4. 
\end{equation}
There exists a null set $N_1'$ containing $E$ such that 
\begin{equation}\label{eq n1}
\tilde{L}_f(x,x_0)=L^{N_1'}_f(x, x_0)=\inf \{I_f(\gamma): \gamma\in {\Gamma}^{N_1'}(x,x_0)\}.
\end{equation}
For $\tilde{f}=f+\infty\chi_{N_1'}$, we can apply Theorem \ref{thm doubling} to find a curve transversal to $N_1'$ connecting $x$ and $x_0$. In particular, there exists a curve $\gamma_1$ such that
\begin{equation}\label{eq E est3}
I_f(\gamma_1)\le L^{N_1'}_f(x,x_0)+\frac{\vep}{4}=\tilde{L}_f(x,x_0)+\frac{\vep}{4}<\frac{\vep}{2}.
\end{equation}
Analogously, there exists a curve $\gamma_2$ transversal to a null set $N_2'$ containing $E$ and connecting $y, y_0$ in $\Oba$ such that 
\begin{equation}\label{eq E est4}
\tilde{L}_f(y, y_0)\leq I_f(\gamma_2)<\vep/2.
\end{equation}

Choose any curve $\gamma \in \Gamma^E(x, y)$ and 
let $\xi$ be a curve connecting $x_0$ and $y_0$ defined by joining $\gamma_1$, $\gamma$ and $\gamma_2$. Therefore we obtain
\begin{equation}\label{eq E est}
\begin{aligned}
\tilde{L}_f(x,y)&\le \tilde{L}_f(x,x_0)+\tilde{L}_f(x_0,y_0)+\tilde{L}_f(y_0,y)\\
&\le \vep+I(\xi)= \vep+I_f(\gamma_1)+I_f(\gamma_2)+I_f(\gamma)\le I_f(\gamma)+2\vep.
\end{aligned}
\end{equation}
Since $\vep$ is arbitratry, it follows that 
\[
\tilde{L}_f(x,y)\leq \inf \{I_f(\gamma): \gamma\in {\Gamma}_{x,y}^{E}\}=L_f^E(x, y)
\] 
for all $x, y\in \Oba$. We have completed the proof of \eqref{sup attain1} under the assumption (A1).

If, on the other hand, (A2) holds, then there exists a null set 
such that \eqref{ess bound} holds. We define $E=\bigcup_{x_0,y_0\in D} (N_{x_0y_0}\cup N_f)$ in this case. It is clear that $E$ is still a null set.  For every $\vep>0$, let $\vep'=\vep/(2C\beta)$. Then for any $x, y\in \Oba$ and $x_0\in D\cap B_{\vep'}(x)$ and $y_0\in D\cap B_{\vep'}(y)$, there exists a null set $N_1'$ containing $E$ such that \eqref{eq n1} holds also in this case. 
Thanks to the $\infty$-Poincar\'e inequality, we adopt \eqref{eq prop regular 2} in Proposition \ref{prop regular 2} to get \eqref{eq E est2}. The rest of the proof is the same as that under the assumption (A1). We find curves $\gamma_1$ and $\gamma_2$ satisfying estimates \eqref{eq E est3} and \eqref{eq E est4}, and build $\xi$ again by concatenating $\gamma_1, \gamma$ and $\gamma_2$. The same estimate as in \eqref{eq E est} yields \eqref{sup attain1} in this case.
\end{proof}

\subsection{Transversal Monge solutions}

Using the optical length $\tilde{L}_f$ (or equivalently $L^E_f$ with $E$ given in Lemma \ref{E}), we can consider a different notion of Monge solutions, which we call transversal Monge solutions. 

\begin{defi}[Transversal Monge solutions] \label{tran Monge def}
We say that a locally bounded function $u: \Omega\to \R$ is a transversal Monge solution (resp. subsolution, supersolution) to \eqref{eikonal} in $\Omega$ if for any $x_0\in \Omega$
\begin{equation}\label{tran monge}
\limsup_{x\to x_0}\frac{u(x_0)-u(x)}{\tilde{L}_f(x, x_0)}=1\ \ \ (\text{resp.}\ \le,\ge).
\end{equation}
\end{defi}

\begin{rmk}
Under assumption (A2), by \eqref{ol-upper} and \eqref{ol-lower}, we can show that \eqref{tran monge} holds if and only if
\begin{equation}\label{tran monge2}
\limsup_{x\to x_0}\frac{u(x_0)-u(x)-\tilde{L}_f(x,x_0)}{{d}(x,x_0)}=0\ \ \ (\text{resp.}\ \le,\ge).
\end{equation}
Indeed, \eqref{tran monge} is equivalent to 
 \[
 \limsup_{x\to x_0}\frac{u(x_0)-u(x)-\tilde{L}_f(x, x_0)}{\tilde{L}_f(x, x_0)}=0\ \ \ (\text{resp.}\ \le,\ge),
 \]
which is further equivalent to \eqref{tran monge2},
since we have 
 \[
\alpha\leq {\tilde{L}_f(x, x_0)\over d(x, x_0)}\leq C\Vert f\Vert_{L^\infty(\Omega)}.
 \]
 \end{rmk}
 \begin{rmk}\label{rmk cont2}
Following Remark \ref{rmk cont1}, we can show that $|\nabla^- u|(x_0)=f(x_0)$ holds for any transversal Monge solution and any $x_0\in \Omega$ where $f$ is continuous. 
 \end{rmk}

We can repeat our arguments in the previous section, replacing $L_f$ by $\tilde{L}_f$, to prove the uniqueness and existence of transversal Monge solutions of \eqref{eikonal}\eqref{dirichlet}. Note that thanks to (A1) or (A2), $(\Oba, \tilde{L}_f)$ is bounded if $(\Omega, d)$ is bounded.  Below we give the statements without proofs.




\begin{thm}\label{cor comparison2}
Let $(\X, d, \mu)$ be a complete geodesic space with $\mu$ doubling and let $\Omega\subsetneq \X$ be a bounded domain. Assume that $f$ satisfies \eqref{f-lower} and that either (A1) or (A2) holds. 
Let $u\in USC_{\tilde{L}}(\Oba)$ and $v\in LSC_{\tilde{L}}(\overline{\Omega})$ be respectively a Monge subsolution and a Monge supersolution to \eqref{eikonal} in the sense of Definition~\ref{tran Monge def}. If $u\le v$ on $\partial \Omega$, then $u\le v$ in $\Oba.$
\end{thm}

\begin{thm}\label{thm exist tran}
Let $(\X, d, \mu)$ be a complete geodesic space with $\mu$ doubling and let $\Omega\subsetneq \X$ be a bounded domain. 
Assume that $f$ satisfies \eqref{f-lower} and that either (A1) or (A2) holds. Let $g: \partial \Omega\to \R$ be bounded,
and let $\tilde{u}$ be defined by \eqref{lax2}, where $\tilde{L}_f$ is given by \eqref{olf}. Then $\tilde{u}$ is a  transversal Monge solution of \eqref{eikonal}, which is Lipschitz continuous with respect to $\widetilde{L}_f$ and uniformly continuous with respect to $d$ in $\Oba$. Moreover, if $g$ satisfies  
\begin{equation}\label{growth2}
g(x)\le  \tilde{L}_f(x, y)+g(y)\quad \text{for all $x, y\in \partial \Omega$, }
\end{equation}
then $\tilde{u}$ is the unique transversal Monge solution of \eqref{eikonal}\eqref{dirichlet}. 
\end{thm}

\begin{rmk}
Recall that a metric space is said to be proper if every closed and bounded set is compact. A complete metric measure space $(X,d,\mu)$ with $\mu$ doubling is proper \cite[Lemma 4.1.14]{HKSTBook}. Furthermore, Hopf-Rinow Theorem implies that a complete and proper length space is a geodesic space. 
\end{rmk}

The Lipschitz continuity of $\tilde{u}$ can be obtained via the relation 
\begin{equation}\label{tildeLip}
\tilde{u}(x)\leq \tilde{u}(y)+\tilde{L}_f(x, y) \quad \text{for all $x, y\in \Oba$,}
\end{equation}
which is a counterpart of \eqref{dpp-sub} in the current setting. 

By \eqref{curve compare1}, we see that in general the Monge solution $u$ and transversal Monge solution $\tilde{u}$ of \eqref{eikonal}\eqref{dirichlet} satisfy $u\leq \tilde{u}$ in $\Oba$ under the same given boundary data. 

\subsection{Maximal weak solutions}

In the Euclidean space, for bounded continuous inhomogeneous term $f$, one can define weak solutions of \eqref{eikonal} by requiring the function to be locally Lipschitz and satisfies the equation almost everywhere in 
$\Omega$. Such kind of solutions are also called Lipschitz a.e. solutions in the literature. 
We extend this notion to metric measure spaces for a possibly discontinuous $f$.

\begin{defi}
We say that $u\in {\rm Lip_\loc}(\Omega)\cap C(\Oba)$ is a weak solution (resp. subsolution, supersolution) to  \eqref{eikonal} if the slope $|\nabla u|(x)$ exists and  satisfies $|\nabla u|(x)= f(x) $ (resp. $\le $, $\ge $) at almost every $x\in \Omega$. 
\end{defi}



When the measure $\mu$ on $\Oba$ is doubling and supports an $\infty$-Poincar\'e inequality, we know 
from~\cite{DJ} that $u$ is Lipschitz continuous with respect to the metric $d$ and that
$|\nabla u|$ is the least $\infty$-weak upper gradient of $u$. Our construction of $u$ in this setting
yields $|\nabla u|\le f$ and in general we may not have equality. 
However, if $f\in L^\infty(\overline{\Omega})$ is continuous almost everywhere, we do obtain that $f=|\nabla u|$. This
is the focus of this subsection. 

Let us denote the class of weak subsolutions to \eqref{eikonal},\,\eqref{dirichlet} by 
\begin{equation}\label{weaksub}
S_{weak}:=\{v\in {\rm Lip_\loc}(\Omega)\cap C(\Oba)
:\ \text{$|\nabla v(x)|\le f(x)$ a.e. in $\Omega$ and $v\le g$ on $\partial\Omega$}\}.
\end{equation}

We can show that the maximal weak subsolution 
is the transversal Monge solution provided that $f$ is upper semicontinuous almost everywhere and the metric space $\X$ satisfies the $\infty$-weak Fubini property. 
A metric measure space $(\X,d,\mu)$ is said to satisfy the 
$\infty$-weak Fubini property if for any null set $N$ and $\vep>0$, given any distinct points $x,y\in \X$, there exists a curve connecting $x,y$ transversal to $N$ such that $\ell(\gamma)\le d(x,y)+\vep$. In particular, a space satisfying the $\infty$-weak Fubini 
necessarily supports an $\infty$-Poincar\'e inequality. More discussion on the $\infty$-weak Fubini property can be found in \cite[Section 4]{DJS2}.

\begin{prop}\label{maxLip}
Let $(\X, d, \mu)$ be a complete geodesic space with $\mu$ a doubling measure.  Suppose that 
$\Omega\subsetneq \X$ is a bounded domain with $(\overline{\Omega},d,\mu)$ satisfying the $\infty$-weak Fubini property. 
Assume that $f\in L^\infty(\overline{\Omega})$ and that $f$ satisfies \eqref{f-lower}. Let $g: \partial \Omega\to \R$ satisfy \eqref{growth2} and $\tilde{u}$ be the transversal Monge solution defined in \eqref{lax2}.
Then, $v\leq \tilde{u}$ in $\overline{\Omega}$ holds for all $v\in S_{weak}$.
If in addition $f$ is assumed to be upper semicontinuous almost everywhere, i.e., 
there exists a set $N\subset \Omega$ with $\mu(N)=0$ such that $f|_{{\Omega}\setminus N}$ is upper semicontinuous, 
then $\tilde{u}$ is a weak subsolution of \eqref{eikonal},\,\eqref{dirichlet}. In particular, 
$\tilde{u}$ is the maximal weak subsolution in the sense that 
\begin{equation}\label{eq:maximal}
\tilde{u}(x)=\sup\{v(x): v\in S_{weak}\} \quad \text{for all $x\in \Omega$.}
\end{equation}
\end{prop}

\begin{rmk}
If we further assume that $\X$ supports a $p$-Poincar\'e inequality for some $1<p<\infty$, 
then we have $g_u=|\nabla u|$ almost everywhere \cite[Theorem 6.1]{Che}, where $g_u$ 
denotes the least $p$-weak upper gradient of $u$. In particular, the function $u_0$ defined 
in \eqref{lax} is the maximal Lipschitz solution to the Dirichlet problem
\[
  \begin{cases}
g_u(x)=f(x) &\quad\quad \text{a.e. in $\Omega$} \\
 u=g &\quad\quad \text{on $\partial\Omega$,} 
  \end{cases}
\]
where $f,g$ satisfy the conditions in Proposition \ref{maxLip}. 
\end{rmk}

\begin{proof}[Proof of Proposition \ref{maxLip}]
We have seen in Theorem \ref{thm exist tran} that $\tilde{u}$ is uniformly continuous in $\Oba$ with respect to the metric $d$. Let us show that $v\leq \tilde{u}$ in $\Omega$ for every $v\in S_{weak}$. Take the null set $N_v=\{x\in \Omega: |\nabla v|(x)> f(x)\}$. 
Note that $|\nabla v|(x)$ is an upper gradient of $v$, i.e., 
\[
|v(x)-v(y)|\leq \int_\gamma |\nabla v|\, ds
\]
holds for every curve $\gamma$ joining $x$ and $y$ in $\Oba$; see \cite[Lemma 6.2.6]{HKSTBook}. 
 Let $N'=N_v\cup E$, where $E$ is the null set given in Lemma \ref{E}.
 It is clear that $N'$ is still a null set.  We fix $y\in\partial\Omega$.
 For any $\vep>0$, we can choose a curve $\gamma$ in $\overline{\Omega}$  connecting $x$ and 
$y$ such that it is transversal to $N'$ and
\begin{equation}\label{eq max weak}
\int_\gamma f\ ds+g(y)\le \tilde{u}(x)+\vep.
\end{equation}
This is possible because $\Oba$ supports the $\infty$-weak Fubini property.
Due to the transversality, we get
\[
\int_\gamma |\nabla v|\ ds\le \int_\gamma f\ ds.
\]
Hence, 
 it follows that 
\[
\begin{aligned}
v(x)=v(x)-v(y)+v(y)&\le   |v(x)-v(y)|+g(y)\\
&\le  \int_0^{\ell} |\nabla v|\, ds+g(y)\le \int_\gamma f\ ds+g(y). 
\end{aligned}
\] 
By \eqref{eq max weak}, we thus have $v(x)\le \tilde{u}(x)+\vep$.
Since $\vep$ is arbitrary, it follows that $v(x)\le \tilde{u}(x)$ for every $x\in \Omega$.

We next prove that $\tilde{u}$ is a weak solution of \eqref{eikonal} under the upper semicontinuity of $f\vert_{\Omega\setminus N}$ for a null set $N$. Let $N_\ast=N\cup E$ and fix $x_0\in \Omega\setminus N_\ast$ arbitrarily.  In view of Remark \ref{rmk cont2}, $\tilde{u}$ satisfies $|\nabla \tilde{u}|(x_0)\geq |\nabla^- \tilde{u}|(x_0)=f(x_0)$. 
In what follows we show that $|\nabla \tilde{u}|(x_0)\leq f(x_0)$.  Then by the upper semicontinuity of $f\vert_{\Omega\setminus N}$, for every $0<\vep<1$ we can take $r>0$ small such that
\begin{equation}\label{eq max weak2}
f(y)\leq f(x_0)+\vep \quad \text{for all $y\in B_{2r}(x_0)\setminus N_\ast$}
\end{equation}
 Using the $\infty$-weak Fubini property of $\Oba$, for any $x\in B_r(x_0)$ it is possible to connect $x_0$ and $x$ by a curve $\gamma$ in $\Omega$ that is transversal to $N_\ast$ and satisfies $\ell(\gamma)\leq (1+\vep)d(x, x_0)<2r$. This implies that $d(x_0, y)<2r$ for any $y\in \gamma$; in other words, $\gamma$ lies in $B_{2r}(x_0)$. It follows from \eqref{eq max weak2} and the transversality of $\gamma$ to $N_\ast$ that 
\[
\int_{\gamma} f\, ds\leq (f(x_0)+\vep) \ell(\gamma)\leq (f(x_0)+\vep)(1+\vep)d(x, x_0).
\]
Applying \eqref{sup attain1} in Lemma \ref{E}, we therefore obtain
\[
\tilde{L}_f(x, x_0)=L_f^E(x, x_0)\leq L_f^{N_\ast}(x, x_0)\leq (f(x_0)+\vep)(1+\vep)d(x, x_0)
\]
for all $x\in \Omega$ with $d(x, x_0)>0$ small.

In view of the Lipschitz continuity of $\tilde{u}$ in \eqref{tildeLip}, we get
\[
|\tilde{u}(x)-\tilde{u}(x_0)|\le (f(x_0)+\vep)(1+\vep)d(x,x_0).
\]
Dividing the inequality by $d(x, x_0)$, letting $d(x, x_0)\to 0$ and then sending $\vep\to 0$, we end up with $|\nabla \tilde{u}|(x_0)\le f(x_0)$. 
Since $x_0\in \Omega\setminus N_\ast$ is arbitrary and $N_\ast$ is a null set, we see that $\tilde{u}$ is a weak subsolution of \eqref{eikonal}. 
Note that $\tilde{u}\leq g$ holds on $\partial \Omega$ thanks to the definition \eqref{lax2}. Hence, $\tilde{u}\in S_{weak}$. Combining this with the first part of our result, we are led to \eqref{eq:maximal}.
\end{proof}

\begin{rmk}\label{rmk weak sol}
We can define the class of weak supersolution and solution by replacing ``$\le$" in \eqref{weaksub} by $``\ge"$ and $``="$ respectively. If $f\vert_{\Omega\setminus N}$ is further assumed to be continuous, then by Remark \ref{rmk cont2}, 
$\tilde{u}$ satisfies $|\nabla^- \tilde{u}|(x_0)=f(x_0)$ for any $x_0\in \Omega\setminus N_\ast$. 
It immediately follows that $|\nabla \tilde{u}|(x_0)\geq f(x_0)$ due to the property $|\nabla\tilde{u}|\ge |\nabla^-\tilde{u}|$. In other words, $\tilde{u}$ is a weak supersolution of \eqref{eikonal} since $\tilde{u}\ge g$ follows from \eqref{growth2}.  Since $\tilde{u}$ is also a weak subsolution, we see that $\tilde{u}$ is a weak solution. 
\end{rmk}

\begin{rmk}
If instead of the condition $f\in L^\infty(\Oba)$, we impose (A1), that is, $\Oba$ satisfies a $p$-Poincar\'e inequality and $f\in L^p(\Oba)$ for finite $p>\max\{1,Q\}$, then the problem becomes more challenging. One  may replace the definition of $S_{weak}$ in \eqref{weaksub} by
\[
\begin{aligned}
S_{weak}:=\{v\in C(\Oba)
:\ &v\le g\text{ on }\partial\Omega,\, \text{ there exists a null set $N_v$ such that $f$} \\
&\text{is the upper gradient of $v$ along all curves transversal to $N_v$}\}.
\end{aligned}
 \]
Then the same conclusion as in Proposition \ref{maxLip} follows from a similar argument.

\end{rmk}

\bibliographystyle{abbrv}

\end{document}